\documentclass[11pt,a4]{amsart}
\pdfoutput=1
\usepackage{amsmath}
\usepackage{amssymb} 
\usepackage{amscd} 
\usepackage[dvipdf]{graphicx} 
\usepackage[dvips]{color}
\usepackage[all]{xy} 
\usepackage{xcolor}

\usepackage{hyperref}
\hypersetup{
	colorlinks=true,
	linktoc=all,
    	linkcolor={red!50!black},
   	citecolor={blue!50!black},
  	urlcolor={blue!80!black}
	} 

\definecolor{dg}{rgb}{0.1,0.4,0.1}
\renewcommand{\labelenumi}{ $(\arabic{enumi})$ }

\newtheorem{theorem}{Theorem}[section]
\newtheorem{lemma}[theorem]{Lemma}

\newtheorem{proposition}[theorem]{Proposition}
\newtheorem{corollary}[theorem]{Corollary}
\newtheorem{claim}[theorem]{Claim}

\newtheorem{remark}[theorem]{Remark}
\newtheorem{question}[theorem]{Question}

\theoremstyle{definition}
\newtheorem{definition}[theorem]{Definition}
\newtheorem{example}[theorem]{Example}

\newtheorem*{thm_hyperbolic_persistent_[G,G]_noncyclic}{Theorem~\ref{hyperbolic_persistent_[G,G]_noncyclic}}
\newtheorem*{lemma_separation}{Lemma~\ref{separation}}
\newtheorem*{lemma_S_K_intersection}{Lemma~\ref{S_K_intersection}}

\newtheorem*{thm_realization}{Theorem~\ref{realization}}

\newtheorem*{thm_non_rigid}{Theorem~\ref{non_rigid}}
\newtheorem*{claim:alpha_non-commute}{Claim~\ref{alpha_non-commute}}

\numberwithin{equation}{section}
\numberwithin{figure}{section}
\numberwithin{table}{section}

\renewcommand{\)}{\textup{)}}

\begin{document}
\baselineskip 13pt

\title[Dehn filling and the knot group]{Dehn filling and the knot group I: \\Realization Property}

\author[T.Ito]{Tetsuya Ito}
\address{Department of Mathematics, Kyoto University, Kyoto 606-8502, JAPAN}
\email{tetitoh@math.kyoto-u.ac.jp}

\author[K. Motegi]{Kimihiko Motegi}
\address{Department of Mathematics, Nihon University, 
3-25-40 Sakurajosui, Setagaya-ku, 
Tokyo 156--8550, Japan}
\email{motegi.kimihiko@nihon-u.ac.jp}

\author[M. Teragaito]{Masakazu Teragaito}
\address{Department of Mathematics Education, Hiroshima University, 
1-1-1 Kagamiyama, Higashi-Hiroshima, 739--8524, Japan}
\email{teragai@hiroshima-u.ac.jp}

\dedicatory{Dedicated to Cameron McA. Gordon on the occasion of his 80th birthday}

\begin{abstract}    
Given a non-trivial knot $K$ in $S^3$, 
its exterior $E(K)$ admits infinitely many Dehn fillings parametrized by slopes $r \in \mathbb{Q}$, 
each of which yields a closed $3$--manifold $K(r)$ and trivializes some elements in the knot group $G(K) = \pi_1(E(K))$ via the induced homomorphism from $G(K)$ onto $\pi_1(K(r))$.   
For hyperbolic knots $K$, it is known that every non-trivial element of $G(K)$ remains non-trivial for all but finitely many Dehn fillings. 
We address the question: 
Given finitely many slopes $r_1, \dots, r_n$, 
does there exist an element in $G(K)$ such that it becomes trivial after $r_i$--Dehn fillings for these pre-specified slopes $r_1, \dots, r_n$, 
while it remains non-trivial for all other non-trivial Dehn fillings? 
In this article, 
we answer this question in the positive for most hyperbolic knots, including those without exceptional surgeries. 
Furthermore, we show that there are infinitely many such elements up to conjugacy and powers. 

\end{abstract}

\maketitle
\tableofcontents
\section{Introduction}
\label{Introduction}

\subsection{Background}
\label{background}
In \cite{Dehn}, 
Dehn considered the following method for constructing $3$--manifolds: 
remove a solid torus neighborhood $V = N(K)$ of a knot $K$ in the $3$--sphere $S^3$ to obtain the exterior $E(K)$ of $K$, then 
glue the solid torus $V$ to $E(K)$ along their boundaries differently. 
The latter half is called {\em Dehn filling} of $E(K)$. 
The gluings are parametrized by slopes defined below.  
Denote by $G(K)$ the {\em knot group} $\pi_1(E(K))$. 
Let $(\mu, \lambda)$ be a preferred meridian-longitude pair of $K$, 
which generates the {\em peripheral subgroup} $P(K) = i_{*}(\pi_1(\partial E(K))) \subset G(K)$. 
Then every simple closed curve on $\partial E(K)$ (based at a base point of $G(K)$) represents $\mu^p \lambda^q$ for some coprime integers $p$ and $q$; we call $\mu^p \lambda^q$ a {\em slope element} of {\em slope} $p/q$. 
(Two slope elements $\mu^p \lambda^q$ and $\mu^{-p} \lambda^{-q}$ have the same slope $p/q$.)
If the above gluing is performed so that the meridian of $V$ has a slope $p/q$, 
then such a Dehn filling is specified as $p/q$--Dehn filling of $E(K)$. We denote the resulting closed $3$--manifold by $K(p/q)$. 

Geometric aspects of Dehn fillings have been extensively studied by many authors along a slogan
``generically, the topology of $E(K)$ persists in $K(r)$''.
See survey article \cite{Go_ICM,Go_Warsaw} and references therein. 
In \cite{IMT_Magnus} the authors study a group theoretic aspect of Dehn filling. 
The purpose of this article is to make a significant improvement and development of \cite{IMT_Magnus}. 

Each $p/q$--Dehn filling induces a natural epimorphism 
\[
p_r \colon G(K) \to  \pi_1(K(p/q)) = G(K) / \langle\!\langle  \mu^p \lambda^q \rangle\!\rangle, 
\]
where  $\langle\!\langle  \mu^p \lambda^q \rangle\!\rangle$ is the normal closure of $\mu^p \lambda^q$ in $G(K)$. 
Since $\gamma = \mu^p \lambda^q$ and its inverse $\gamma^{-1}$ have the same normal closure, 
$\langle\!\langle \gamma^{-1} \rangle\!\rangle = \langle\!\langle \gamma \rangle\!\rangle$ will be denoted by 
$\langle\!\langle r \rangle\!\rangle$ using their slopes $r = p/q \in \mathbb{Q}$. 
Then $g \in G(K)$ becomes trivial after $r$--Dehn filling, i.e. $p_r(g) = 1 \in \pi_1(K(r))$ if and only if $g \in \langle\!\langle r \rangle\!\rangle$. 

The Property P Conjecture, settled by Kronheimer and Mrowka \cite{KM}, 
asserts that $\pi_1(K(r))$ is non-trivial for all $r \in \mathbb{Q}$. 
Hence, 
$\langle\!\langle r \rangle\!\rangle = \langle\!\langle \infty \rangle\!\rangle = G(K)$ if and only if $r = \infty$.
In \cite[Theorem~1.2]{IMT_Magnus}, using Property P we extend this to the peripheral Magnus property:  
$\langle\!\langle r \rangle\!\rangle = \langle\!\langle r' \rangle\!\rangle$ if and only if $r = r'$, where $r, r' \in \mathbb{Q} \cup \{ \infty \}$. 

\begin{definition}
\label{D}
Let $K$ be a non-trivial knot in $S^3$. 
To each element $g \in G(K)$ assigning a subset 
\begin{align*}
\mathcal{S}_K(g) 
& =  \{ r \in \mathbb{Q} \mid p_r(g) = 1 \in \pi_1(K(r)) \}\\
& = \{ r \in \mathbb{Q} \mid g \in \langle\!\langle r \rangle\!\rangle \}
\end{align*}

we obtain a set valued function 
\[
\mathcal{S}_K \colon G(K) \to 2^{\mathbb{Q}}. 
\]
\end{definition}

\medskip

Since $G(K)$ is countable, 
$\mathcal{S}_K \colon G(K) \to 2^{\mathbb{Q}}$ is not surjective. 
So not every subset appears as $\mathcal{S}_K(g)$. 
More precisely, 
$\mathcal{S}_K(g)$ is known to be finite except when $g = 1$ for hyperbolic knots $K$ \cite{Osin,GM,IchiMT}. 
On the other hand, $|\mathcal{S}_K(g)|$ may be arbitrarily large. 
Indeed \cite[Theorem~1.4]{IMT_Magnus} shows that for any finite subset $\mathcal{R}$, 
there exists a non-trivial element $g$ with $\mathcal{S}_K(g) \supset \mathcal{R}$. 
However, this result says nothing if $g$ remains non-trivial in $\pi_1(K(r))$ for slopes $r \in \mathbb{Q} - \mathcal{R}$. 
Sharpening this, we would like to ask: 

\begin{question}
\label{motivation}
Given a finite subset $\mathcal{R} \subset \mathbb{Q}$, 
is it possible to take an element $g \in G(K)$ so that $\mathcal{S}_K(g)$ is exactly the same as $\mathcal{R}$?  
Furthermore, if it does exist, then how may such elements can we have? 
\end{question}

\begin{example}
\label{|R|=0,1}
\begin{enumerate}
\renewcommand{\labelenumi}{(\arabic{enumi})}
\item
Since the meridian $\mu$ normally generates $G(K)$, 
Property P is equivalent to saying that $p_r(\mu) \ne 1 \in \pi_1(K(r))$ for all $r \in \mathbb{Q}$, 
i.e. $\mathcal{S}_K(\mu) = \emptyset$. 
\item
For any slope element $\gamma$ representing slope $r \in \mathbb{Q}$, 
$\mathcal{S}_K(\gamma) = \{ r \}$ if $K$ has no finite surgery slope; see Proposition~\ref{slope}. 
Here a slope $r \in \mathbb{Q}$ is a {\em finite surgery slope} if $\pi_1(K(r))$ is finite. 
\end{enumerate}
\end{example}

In Example~\ref{|R|=0,1} (2), 
the absence of finite surgery slope cannot be removed.  

\begin{example}[Non realizable finite subset]
\label{pretzel_non_separation}
Let $K$ be the $(-2, 3, 7)$--pretzel knot. 
Take a finite surgery slope $18$. 
Then as shown in \cite[Example 6.2]{IMT_Magnus}, 
$\langle \! \langle \frac{18}{5} \rangle\!\rangle \subset \langle \! \langle 18 \rangle\!\rangle$. 
Hence if $\mathcal{S}_K(g)$ contains the slope $\frac{18}{5}$, 
then it necessarily contains the slope $18$ as well. 
This means that there is no element $g$ with $\mathcal{S}_K(g) = \{ \frac{18}{5} \}$.  
More generally, any finite subset $\mathcal{R}$ which contains $\frac{18}{5}$, but does not contain $18$ cannot be 
realized by $\mathcal{S}_K(g)$. 
\end{example}

This example may be generalized to knots with finite surgery slopes. 
For details, see \cite[Proposition~6.1]{IMT_Magnus}.  

\medskip

A finite surgery slope is a Seifert surgery slope, i.e.  $K(r)$ is a Seifert fiber space due to the Geometrization Theorem \cite{Pe1,Pe2,Pe3}. 
For a finite surgery slope $r$, every element $g \in G(K)$ becomes a torsion element in $\pi_1(K(r))$. 
We say that $r$ is a {\em torsion surgery slope} if $\pi_1(K(r))$ has a torsion element. 
As we will observe later (Lemma~\ref{torsion_slope}), 
a slope $r$ is a torsion surgery slope if and only if it is either a finite surgery slope or a {\it reducing surgery slope} 
meaning that $K(r)$ is reducible.
(In the latter case, since $K$ is non-trivial, 
$K(r)$ should be non-prime \cite{GabaiIII}, 
and following \cite{GLu2} $K(r)$ has a non-trivial lens space summand.)
In fact, the Cabling Conjecture  \cite{GS} asserts that any hyperbolic knot has no reducing surgery slope. 
At the present it is known that a hyperbolic knot has at most two such slopes \cite{GLreducible}. 

\medskip

\subsection{Results}
\label{results}

Recall that for hyperbolic knots $K$, $\mathcal{S}_K(g)$ is always finite for all non-trivial knots $g \in G(K)$.  
Example~\ref{pretzel_non_separation} shows that not every finite subset $\mathcal{R}$ can be realized by  $\mathcal{S}_K(g)$. 
Note that in this example, the complement of $\mathcal{R}$ contains finite (Seifert) surgery slope. 
The theorem below answers the first part of Question~\ref{motivation} in the positive for most hyperbolic knots. 
In the following a finite subset $\mathcal{R}$ may be the emptyset. 

\begin{theorem}[Realization Theorem]
\label{realization}
Let $K$ be a hyperbolic knot and let $\mathcal{R} = \{ r_1, \ldots, r_n\}$ be any finite subset of $\mathbb{Q}$. 
Then we may find an element $g \in [G(K), G(K)] \subset G(K)$ so that  $\mathcal{S}_K(g) = \mathcal{R}$ 
whenever the complement of $\mathcal{R}$ contains neither a Seifert surgery slope nor two reducing surgery slopes.   
\end{theorem}

More precisely, in the case where $\mathcal{R} = \emptyset$, there exists $g \in [G(K), G(K)]$ such that $\mathcal{S}_K(g) = \emptyset$ if and only if 
$K$ has no cyclic surgery; see Theorem~\ref{hyperbolic_persistent_[G,G]_noncyclic}.  

Theorem~\ref{realization} immediately implies: 

\begin{corollary}
\label{no_Seifert}
Let $K$ be a hyperbolic knot which admits neither a Seifert surgery nor two reducing surgeries. 
Then any finite subset $\mathcal{R}$ is realized by $\mathcal{S}_K(g)$ for some element $g$. 
\end{corollary}

\medskip

Thus, for any hyperbolic knot without {\em exceptional} (i.e. non-hyperbolic) surgery, 
every finite subset  is realized by $\mathcal{S}_K(g)$ for some element $g$. 

Such hyperbolic knots are common. 
Indeed, every knot can be converted into such a hyperbolic knot by a single crossing change; 
moreover for a given knot there are infinitely many such crossing changes \cite{MM_crossing}. 

\medskip

\begin{example}
Let $K$ be a knot depicted by Figure~\ref{knot_without_exceptional}, which is the simplest hyperbolic knot  $6_3$ without exceptional surgery \cite{BrittenhamWu}. 
Then every finite subset $\mathcal{R} \subset \mathbb{Q}$ is realized as 
$\mathcal{S}_K(g)$ for some element $g \in G(K)$.
For instance, let $a_n$ be the rational number up to $n$-th decimal place of $\pi$, 
and put 
\[\mathcal{R}_n = \{ 3,\ 3.1,\ 3.14,\ 3.141,\ 3.1415,\; \dots,\; a_n \}.\]
Then for any positive integer $n$ we have an element $g_n$ which satisfies 
$\mathcal{S}_K(g_n) = \mathcal{R}_n$. 
\end{example}
\begin{figure}[htb]
\centering
\includegraphics[width=0.2\textwidth]{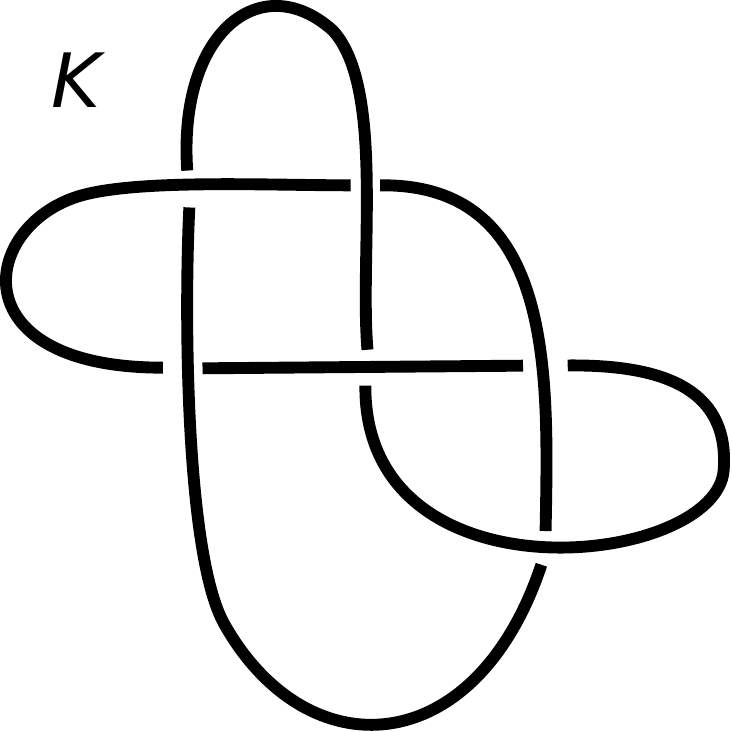}
\caption{The simplest hyperbolic knot (with respect to crossing numbers) $6_3$ without exceptional surgery} 
\label{knot_without_exceptional}
\end{figure}

\medskip

If $\mathcal{R}$ is the set of exceptional surgery slopes, 
then it is a (generically empty) finite subset of $\mathbb{Q}$ \cite{T1,T2},  
and by definition it satisfies the condition of Theorem~\ref{realization}. 
Hence we have 

\begin{corollary}
\label{exceptional_slopes}
Let $K$ be a hyperbolic knot. 
Then the family of the exceptional surgery slopes $\mathcal{R}$ is realized by $\mathcal{S}_K(g)$ for some element $g$. 
\end{corollary}

\medskip

We say that a non-trivial element $g$ is \textit{peripheral} if it is conjugate into the {\em peripheral subgroup} $P(K) = i_*(\pi_1(\partial E(K))) \subset G(K)$.

\begin{remark}
\label{Realization_theorem_non_peripheral}
The element $g$ in Theorem~\ref{realization} and Corollaries~\ref{no_Seifert} and \ref{exceptional_slopes} can be taken as a non-peripheral element of 
$[G(K), G(K)]$; 
see Remark~\ref{realization_non-peripheral}. 
\end{remark}

Let $K$ be a hyperbolic knot without finite surgery. 
Then for any peripheral element $g$, 
$\mathcal{S}_K(g)$ consists of a single slope; see Proposition~\ref{slope}. 
If $\mathcal{S}_K(g)$ consists of a single slope, then is $g$ peripheral? 
Theorem~\ref{realization} (Corollary~\ref{no_Seifert}) and Remark~\ref{Realization_theorem_non_peripheral} show that 
the function $\mathcal{S}_K \colon G(K) \to 2^{\mathbb{Q}}$ does not detect peripherality of elements, 
i.e. there exists a non-peripheral element $g \in G(K)$ for which $\mathcal{S}_K(g)$ consists of a single slope.

The next result implies that the finite subset $\mathcal{R} \subset \mathbb{Q}$ in Theorem~\ref{realization} is realized by infinitely many elements. 

\begin{theorem}
\label{non_rigid}
Let $K$ be a hyperbolic knot without torsion surgery slope. 
For any non-peripheral element $g \in [G(K), G(K)]$ there are infinitely many, mutually non-conjugate elements 
$\alpha_m \in G(K)$ which enjoy the following. 
\begin{enumerate}
\item
$\mathcal{S}_K(\alpha_m) = \mathcal{S}_K(g)$, 
\item
$\alpha_m$ is not conjugate to any powers of $g$. 
\end{enumerate}
\end{theorem}

This theorem answers the second part of Question~\ref{motivation} for hyperbolic knots without torsion surgery slope. 

\begin{remark}
\label{Realization_infinite}
Applying Theorem~\ref{non_rigid} to the element $g$ in Theorem~\ref{realization} \(with Remark~\ref{Realization_theorem_non_peripheral}\), 
we obtain infinitely many, mutually non-conjugate elements 
$\alpha_m$ satisfying $(1)$ $\mathcal{S}_K(\alpha_m) = \mathcal{S}_K(g) = \mathcal{R}$, and $(2)$ $\alpha_m$ is not conjugate to any powers of $g$. 
\end{remark}

\medskip

Suppose that $K$ is a hyperbolic knot. 
Then since  $\mathcal{S}_K(g)$ is finite for all non-trivial elements $g$, 
we see that $\displaystyle \bigcap_{r \in \mathcal{R}}   \langle\!\langle r \rangle\!\rangle = \{ 1 \}$ whenever $\mathcal{R}$ is an ``infinite'' subset of $\mathbb{Q}$. 
On the other hand, 
Corollary~\ref{no_Seifert} has the following paraphrase. 

\begin{corollary}[Structure of normal closures of slopes]
\label{realization_knot_group}
Let $K$ be a hyperbolic knot without Seifert surgery nor reducing surgery. 
Let $\mathcal{R}$ be any finite, non-empty subset of $\mathbb{Q}$. 
Then there exists a non-trivial element $g \in G(K)$ such that 
\[
g \in \bigcap_{r \in \mathcal{R}}   \langle\!\langle r \rangle\!\rangle,\quad \textrm{but}\quad g \not\in \bigcup_{s \in \mathbb{Q} - \mathcal{R}} \langle\!\langle s \rangle\!\rangle.
\]
\end{corollary}

\medskip

\subsection{Organization of the paper}
\label{organization}

In Section~\ref{function_S_K} we will collect elementary properties of the function $\mathcal{S}_K \colon G(K) \to 2^{\mathbb{Q}}$, 
and then in Section~\ref{normal_closure} we will prepare a useful result about the normal closure of slope elements (Proposition~\ref{slope}). 

In the proof of Theorem~\ref{realization} (Realization Theorem), 
as the first step, 
we need to realize the emptyset by $\mathcal{S}_K(g)$ for some element $g \in [G(K), G(K)]$. 
(Note that the meridian $\mu$ satisfies $\mathcal{S}_K(\mu) = \emptyset$, however $\mu \not\in [G(K), G(K)]$.)
Actually, more precisely, 
we will establish the following. 
A slope $r$ is called a {\it cyclic surgery slope} if $\pi_1(K(r))$ is cyclic. 

\begin{theorem}[Elements with $\mathcal{S}_K(g) = \emptyset$]
\label{hyperbolic_persistent_[G,G]_noncyclic}
Let $K$ be a hyperbolic knot in $S^{3}$.  
Then there exist infinitely many, mutually non-conjugate elements $g \in [G(K),G(K)]$ 
such that $p_s(g) \neq 1$ in $\pi_1(K(s))$ for all non-cyclic surgery slopes $s \in \mathbb{Q}$. 
In particular, 
if $K$ has no cyclic surgery slope, then there exist infinitely many, mutually non-conjugate elements $g \in [G(K),G(K)]$ such that 
$\mathcal{S}_K(g) = \emptyset$. 
\end{theorem}

\begin{remark}
\label{persistent_homologically_trivial_not_meridian}
In Theorem~\ref{hyperbolic_persistent_[G,G]_noncyclic}, 
since $g \in [G(K), G(K)]$, i.e. $g$ is homologically trivial, 
it is not conjugate to any power of the meridian of $K$. 
\end{remark} 

If $K$ has a cyclic surgery slope $r$, 
then for any element $g$ in $[G(K),G(K)]$, 
we have $p_r(g) = 1$. 
Theorem~\ref{hyperbolic_persistent_[G,G]_noncyclic} shows that admitting cyclic surgery is the only obstruction for $G(K)$ to have an element $g$ with $\mathcal{S}_K(g) = \emptyset$ that belongs to $[G(K),G(K)]$. 
Following the cyclic surgery theorem \cite{CGLS}, every hyperbolic knot admits at most two non-trivial cyclic surgery slopes. 
Theorem~\ref{hyperbolic_persistent_[G,G]_noncyclic} will be also used in the first step of the proof of Lemma~\ref{separation} (Separation Lemma) below.

Section~\ref{homologically0} is devoted to a proof of Theorem~\ref{hyperbolic_persistent_[G,G]_noncyclic}. 
The proof requires elementary hyperbolic geometry and stable commutator length, 
so we will briefly recall some definitions and prepare useful results in Section~\ref{scl}.
Then in Section~\ref{finite_separation} we proceed to prove the ``Separation Lemma'' below, which is the key step to establish Theorem~\ref{realization}. 

\begin{lemma}[Separation Lemma]
\label{separation}
Let $K$ be a hyperbolic knot.    
Let $\mathcal{R} = \{ r_1, \dots, r_n \}$ and $\mathcal{S} = \{ s_1, \dots, s_m \}$ be any finite, non-empty subsets of $\mathbb{Q}$ such that 
$\mathcal{R} \cap \mathcal{S} = \emptyset$. 
Assume that $\mathcal{S}$ does not contain a Seifert surgery slope.  
Then there exists an element $g \in [G(K), G(K)] \subset G(K)$ 
such that $\mathcal{R} \subset \mathcal{S}_K(g) \subset \mathbb{Q} - \mathcal{S}$. 
\end{lemma}

To bridge the gap between Lemma~\ref{separation} (Separation Lemma) and Theorem~\ref{realization} (Realization Theorem), 
we apply the following lemma, which will be proved in Section~\ref{shrink}. 

\begin{lemma}[Shrinking Lemma]
\label{S_K_intersection}
Let $K$ be a hyperbolic knot. 
Let $g$ be a non-peripheral element and $h$ a non-trivial element in $[G(K),G(K)]$. 
Assume that $\mathbb{Q} - \mathcal{S}_K(g)$ contains neither a finite surgery slope nor two reducing surgery slopes. 
Then there are infinitely many integers $n$ with a constant $N_n > 0$ such that  
\[
\mathcal{S}_K(g) \cap \mathcal{S}_K(h) = \mathcal{S}_K(g^mh^n) 
\] 
for infinitely many integers $m \ge N_n$ for each $n$.
\end{lemma}

\begin{remark}
\label{shrink_torsion_free}
In Lemma~\ref{S_K_intersection}, 
if $K$ has no torsion surgery, 
then skipping the discussion for torsion elements in its proof, 
we may take $n$ as any non-zero integer and a constant $N_n > 0$ so that 
\[
\mathcal{S}_K(g) \cap \mathcal{S}_K(h) = \mathcal{S}_K(g^mh^n) 
\] 
for all integers $m \ge N_n$. 
\end{remark}

Lemma~\ref{S_K_intersection}, together with Remark~\ref{shrink_torsion_free}, 
also plays a key role in the proof of Theorem~\ref{non_rigid} (Section~\ref{identical trivialization}). 

\medskip

\section{Elementary properties of the function $\mathcal{S}_K$}
\label{function_S_K}
Let us collect some elementary properties of the function $\mathcal{S}_K \colon G(K) \to 2^{\mathbb{Q}}$. 

Since $p_r(a^{-1}ga) = 1$ if and only if $p_r(g) = 1$, and 
$p_r(g^{-1}) = 1$ if and only if $p_r(g)=1$, 
we have 

\begin{proposition}
\label{S_K_class_function}
$\mathcal{S}_K \colon G(K) \to 2^{\mathbb{Q}}$ is a class function, 
namely $\mathcal{S}_K(a^{-1}ga) = \mathcal{S}_K(g)$ for all $g$ and $a$ in $G(K)$. 
Furthermore, 
$\mathcal{S}_K(g^{-1}) = \mathcal{S}_K(g)$. 
\end{proposition}

Since $hg = h(gh)h^{-1}$, 
Proposition~\ref{S_K_class_function} immediately implies: 

\begin{proposition}
\label{S_K_commute}
$\mathcal{S}_K(gh) = \mathcal{S}_K(hg)$. 
\end{proposition}

When $K$ is a prime knot, any automorphism $\phi\colon G(K) \rightarrow G(K)$ is induced by a homeomorphism $f$ of $E(K)$ \cite{Tsau2}, 
so either $\phi(\langle\! \langle r \rangle\! \rangle) = \langle\! \langle r \rangle\! \rangle$ or $\phi(\langle\! \langle r \rangle\! \rangle) = \langle\! \langle -r \rangle\! \rangle$. 
The latter case happens only if $K$ is amphicheiral and $f$ reverses the orientation of $E(K)$. 
Since $g \in \langle\! \langle r \rangle\! \rangle$ if and only if $\phi(g) \in \phi(\langle\! \langle r \rangle\! \rangle)$, 
$\mathcal{S}_K(g)$ is almost $\mathrm{Aut}(G(K))$-invariant in the following sense.

\begin{proposition}
If $K$ is prime, for an automorphism $\phi\colon G(K) \rightarrow G(K)$, 
$\mathcal{S}_K(\phi(g)) = \pm \mathcal{S}_K(g)$. Furthermore, if $K$ is not amphicheiral then $\mathcal{S}_K(\phi(g)) = \mathcal{S}_K(g)$. 
Here for $\mathcal{S} =\{s_1,s_2,\ldots \} \subset \mathbb{Q}$, we put $\pm \mathcal{S}=\{\pm s_1, \pm s_2,\ldots,\}$. 
\end{proposition}

By definition, 
$\mathcal{S}_K$ also satisfies the following. 

\begin{proposition}\
\label{cup_cap}
\begin{enumerate}
\renewcommand{\labelenumi}{(\arabic{enumi})}
\item
$\mathcal{S}_K(g) \cap \mathcal{S}_K(h) \subset \mathcal{S}_K(gh)$. 

\item
$\mathcal{S}_K(g) \cup \mathcal{S}_K(h) \subset \mathcal{S}_K([g, h])$.
\end{enumerate}
\end{proposition}

\begin{proof}
(1) Let us take $r \in \mathcal{S}_K(g) \cap \mathcal{S}_K(h)$. 
Then $p_r(g) = p_r(h) = 1$, and thus $p_r(gh) = p_r(g)p_r(h) = 1$, 
which shows $r \in \mathcal{S}_K(gh)$. 

(2) If $r \in \mathcal{S}_K(g) \cup \mathcal{S}_K(h)$, then $p_r(g) = 1$ or $p_r(h) = 1$. 
This implies $p_r([g, h]) = p_r(g)p_r(h)p_r(g)^{-1}p_r(h)^{-1} = 1$, 
hence $r \in \mathcal{S}_K([g, h])$. 
\end{proof}

Now we describe a relation between $\mathcal{S}_K(g)$ and $\mathcal{S}_K(g^n)$. 

\begin{proposition}\
\label{S_K(g)_S_K(g^n)}
\[
\mathcal{S}_K(g) \subset \mathcal{S}_K(g^n) \subset \mathcal{S}_K(g) \cup \{ \text{torsion surgery slopes} \}
\]
for any integer $n \neq 0$. 
\end{proposition}

\begin{proof}
Note that $G(K)$ has no torsion, $g^n \ne 1 \in G(K)$. 
If $p_r(g) =1$, then obviously $p_r(g^n) =1$, which shows $\mathcal{S}_K(g) \subset \mathcal{S}_K(g^n)$. 
Assume that $p_r(g^n) = p_r(g)^n =1$. 
If $p_r(g) = 1$, then $r \in \mathcal{S}_K(g)$, otherwise $p_r(g)$ is a torsion element in $\pi_1(K(r))$. 
Thus $r$ is a torsion surgery slope. 
\end{proof}

Note that torsion surgeries are classified as follows. 

\begin{lemma}
\label{torsion_slope}
Let $K$ be a non-trivial knot. 
A slope $r$ is a torsion surgery slope if and only if it is either a finite surgery slope or a reducing surgery slope.
\end{lemma}

\begin{proof}
Assume that a $3$--manifold $M$ has the fundamental group with non-trivial torsion. 
Then $M$ is spherical (i.e. $\pi_1(M)$ is finite) or $M$ is non-prime (\cite[Lemma~9.4]{Hem}), 
and thus $r$ is either a finite surgery slope or a reducing surgery slope. 
Conversely, if $\pi_1(M)$ is finite, then every non-trivial element is a torsion element, 
and if $\pi_1(M)$ is reducible,
since $K$ is non-trivial, it is non-prime \cite{GabaiII}. 
Furthermore, following \cite{GLu2} $M$ has a lens space summand. 
Hence $\pi_1(M)$ has a torsion element. 
\end{proof}

The cabling conjecture \cite{GS}  of Gonz\'{a}lez-Acu\~na and Short asserts that a reducible $3$--manifold is obtained by Dehn surgery on $K$ only when 
$K$ is a cable knot and the surgery slope is the cabling slope. 
In particular, a hyperbolic knot 
is expected to have no reducing surgery.
So for hyperbolic knots, a torsion surgery may coincide with a finite surgery; 
moreover, any hyperbolic knot admits at most two finite surgeries except for the $(-2, 3, 7)$--pretzel knot, which has exactly three such surgeries \cite[Theorem~1.4]{Ni-Zhang-Finite}. 
Heegaard Floer theory puts strong constraints for $K$ having finite surgeries; 
$K$ must be an L-space knot \cite{OS_lens}, 
consequently $K$ is fibered \cite{Ni} and its Alexander polynomial has very restrictive form \cite{OS_lens}. 

\medskip

\section{Trivialization of slope elements}
\label{normal_closure}

In this section we observe the following proposition, which generalizes \cite[Proposition~2.3]{IMT_Magnus}. 
Their proofs are almost identical, but for completeness we give a proof here. 

\begin{proposition}
\label{slope}
Let $K$ ba a knot in $S^3$ and $\gamma$ a slope element of slope $r$ in $G(K)$. 
Suppose that $\gamma^n$ becomes trivial after $s$--Dehn filling for some integer $n \ne 0$.  
Then $s = r$ or $s$ is a finite surgery slope. 
\end{proposition}

\begin{proof}
Without loss of generality we may assume $n > 0$, and put $r = a/b$ and $s = p/q$.
Our assumption implies 
$\gamma^n = (\mu^a  \lambda^b)^{n} \in \langle\!\langle s \rangle\!\rangle$. 
Then
\[
\mu^{(aq-bp)n} 
= \big((\mu^a  \lambda^b)^{n}\big)^q (\mu^p  \lambda^q)^{-bn} 
\in  \langle\!\langle s \rangle\!\rangle. 
\]
If $aq-bp = 0$, then $p/q = a/b$, i.e. $s = r$. 
So we assume that $aq-bp \ne 0$; 
without loss of generality we may further assume $aq-bp > 0$.  

If $(aq-bp)n = 1$, 
then $\mu \in \langle\!\langle s \rangle\!\rangle$ and 
$p_s(\mu) = 1$ in $\pi_1(K(s))$. 
This cannot happen by Example~\ref{|R|=0,1} (1).

So in the following we assume $m = (aq-bp)n \ge 2$. 
Since $\mu^m \in \langle\!\langle s \rangle\!\rangle$, 
$\langle\!\langle \mu^m \rangle\!\rangle \subset \langle\!\langle s \rangle\!\rangle$  
and we have the canonical epimorphism 
$\varphi\colon G(K) / \langle\!\langle \mu^m \rangle\!\rangle \to G(K) / \langle\!\langle s \rangle\!\rangle$. 
Note that $(\varphi(\mu))^m = \varphi(\mu^m) = 1$ in $G(K) / \langle\!\langle s \rangle\!\rangle$. 
If $\varphi(\mu) = 1 \in G(K) / \langle\!\langle s \rangle\!\rangle$,  
then $\varphi(\mu) = \mu \in  \langle\!\langle s \rangle\!\rangle$ and as above $s = \infty$. 
Thus we may assume $\varphi(\mu) \ne 1$, 
i.e. $\mu$ is a non-trivial torsion element in $G(K)\slash \langle\!\langle s \rangle\!\rangle$.  
Recall that an irreducible $3$--manifold $M$ with infinite fundamental group is aspherical \cite[p.48 (C.1)]{AFW}, 
hence $\pi_1(M)$ is torsion free \cite{Hem}. 
Hence $\pi_1(K(s))$ is finite or $K(s)$ is reducible. 
Accordingly $s$ is a finite surgery slope or a reducing surgery slope. 

Now let us eliminate the second possibility. 
Assume to the contrary that $s$ is a reducing surgery slope. 
As in the above $G(K) / \langle\!\langle s \rangle\!\rangle$ has a non-trivial torsion element, 
hence $K(s) \ne S^2 \times S^1$ and thus $K(s)$ is a connected sum of two closed $3$--manifolds other than $S^3$. 
(In general, 
the result of a surgery on a non-trivial knot is not $S^2 \times S^1$ \cite{GabaiIII}.)
By the Poincar\'e conjecture, they have non-trivial fundamental groups, 
and $G(K)\slash \langle\!\langle s \rangle\!\rangle = A \ast B$ for some non-trivial groups $A$ and $B$. 

As we have seen in the above, 
$\varphi(\mu)$, 
the image of $\mu$ under the canonical epimorphism 
$\varphi \colon G(K) / \langle\!\langle \mu^m \rangle\!\rangle \to G(K) \slash \langle\!\langle s \rangle\!\rangle$ 
is a non-trivial torsion element in $A \ast B$. 
By \cite[Corollary~4.1.4]{MKS}, 
a non-trivial torsion element in a free product $A \ast B$ is conjugate to a torsion element of $A$ or $B$. 
Thus we may assume that there exists $g \in A\ast B$ such that $g^{-1}\varphi(\mu)g \in A$. 

On the other hand, 
since $G(K)$ is normally generated by $\mu$, 
$A \ast B$ is normally generated by $\varphi(\mu)$. 
This implies that $A \ast B$ is normally generated by an element $g^{-1}\varphi(\mu)g \in A$. 
In particular, 
the normal closure $\langle\!\langle A \rangle\!\rangle$ of $A$ in $A \ast B$ is equal to $A \ast B$, 
and $(A \ast B) \slash \langle\!\langle A \rangle\!\rangle = \{ 1 \}$. 
However, $(A \ast B) \slash \langle\!\langle A \rangle\!\rangle = B \neq \{1\}$ \cite[p.194]{MKS}. 
This is a contradiction. 
\end{proof}

\begin{remark}
\label{finitely_generated}
Let $K$ be a non-trivial knot, and let $r \in \mathbb{Q}$ be a slope. 
Then $\langle\!\langle r \rangle\!\rangle$ is finitely generated if and only if $r$ is a finite surgery slope 
\(i.e. $\pi_1(K(r))$ is finite\) or 
$K$ is a torus knot $T_{p, q}$  and $r = pq$. 
See \cite[Theorem~1.3]{IMT_Magnus}. 
\end{remark}

\section{Behavior of stable commutator length under Dehn fillings}
\label{scl}
Let $K$ be a hyperbolic knot in $S^3$. 
To find elements $g$ with $\mathcal{S}_K(g) = \emptyset$ in the commutator subgroup $[G(K), G(K)]$, 
the stable commutator length plays a key role. 
In this section we prepare some useful results.  

\subsection{Hyperbolic length and Dehn fillings}

Let $M$ be a closed hyperbolic $3$--manifold. 
Then a representative loop of any non-trivial element $g$ of $\pi_1(M)$ is freely homotopic to a unique closed geodesic $c_g$. 
We may define the length of $g \in \pi_1(M)$, 
denoted by $\ell_M(g)$, 
to be the hyperbolic length of the corresponding closed geodesic $c_g$. 

Let $K$ be a hyperbolic knot and 
$g$ a non-trivial element in $G(K)$. 
Suppose that $s \in \mathbb{Q}$ is a hyperbolic surgery slope for $K$. 
If $p_s(g)$ is non-trivial in $\pi_1(K(s))$, 
then we can define
the length $\ell_{K(s)}(p_s(g))$. 
Let us define
\[ L(g) = \inf_{s \in \mathbb{Q}}\{ \ell_{K(s)}(p_s(g)) \: | \: 
K(s) \textrm{ is hyperbolic and $p_s(g) \ne 1$}\} \ge 0.\]

In this subsection we give a simple characterization of elements $g \in G(K)$ with $L(g) > 0$. 

For convenience of readers, 
we briefly recall Thurston's hyperbolic Dehn surgery \cite{T1,T2}. 

By the assumption we have a holonomy (faithful and discrete) representation
\[
\rho \colon G(K) \cong \pi_1(S^3 - K) \to \mathrm{Isom}^+(\mathbb{H}^3) = \mathrm{PSL}(2, \mathbb{C}).
\]

We regard $\rho = \rho_{\infty}$ and denote its image by $\Gamma_{\infty} \cong \pi_1(S^3 - K)$. 
Recall that $\rho(\mu)$ and $\rho(\lambda)$ are parabolic elements.

By small deformation of the holonomy representation $\rho$ up to conjugation, 
we obtain representations 
$\rho_{x, y} \colon \pi_1(S^3 - K) \to \mathrm{Isom}^+(\mathbb{H}^3) = \mathrm{PSL}(2, \mathbb{C})$ whose image is 
$\Gamma_{x, y}$. 
(The meaning of $x, y$ will be clarified below.)
Deforming $\rho$ continuously, 
$(x, y)$ varies over an open set $U_{\infty}$ in $S^2 = \mathbb{R}^2 \cup \{ \infty \}$; see the proof of \cite[Theorem~5.8.2]{T1}.
By Mostow's rigidity theorem \cite{Mostow,Prasad} $\rho_{x, y}(\mu)$ and $\rho_{x, y}(\lambda)$ are not parabolic when $(x, y) \ne \infty$. 

We may take a conjugation in $\mathrm{Isom}^+(\mathbb{H}^3) = \mathrm{PSL}(2, \mathbb{C})$ 
so that 
\[
\rho_{x, y}(\mu) \colon (z, t) \mapsto (\alpha_{x, y, \mu}z, |\alpha_{x, y, \mu}|\,t)
\]
and 
\[
\rho_{x, y}(\lambda) \colon (z, t) \mapsto (\alpha_{x, y, \lambda}z, |\alpha_{x, y, \lambda}|\,t)
\] 
for some complex numbers $\alpha_{x, y, \mu}$ and $\alpha_{x, y, \lambda}$.

Since both $\rho_{x, y}(\mu)$ and $\rho_{x, y}(\lambda)$ are very close to 
$\rho(\mu)$ and $\rho(\lambda)$, respectively, 
both $\alpha_{x, y, \mu}$ and $\alpha_{x, y, \lambda}$ are nearly equal to $1$. 

Then for $\rho_{x, y}$, 
we have
\[
x \log \alpha_{x, y, \mu} + y \log \alpha_{x, y, \lambda} = 2\pi i.
\]

The complex translation length $\log \alpha_{x, y, \mu}$ (resp. $\log \alpha_{x, y, \lambda}$) 
of $\rho_{x, y}(\mu)$ (resp. $\rho_{x, y}(\lambda)$) is well-defined up to the sign. 
Note that since  $\rho_{x, y}(\mu)$ and  $\rho_{x, y}(\lambda)$ commute, 
the pair $(\log \alpha_{x, y, \mu},\  \log \alpha_{x, y, \lambda})$ is well-defined up to sign. 
If we write 
\[
\alpha_{x, y, \mu} = r_{x, y, \mu} e^{i \theta_{x, y, \mu}},\quad 
\alpha_{x, y,\lambda} = r_{x, y, \lambda} e^{i \theta_{x, y, \lambda}},
\] 
the above equation implies
\[
\begin{cases} 
x \log r_{x, y, \mu} + y \log r_{x, y, \lambda} = 0\\ 
x \theta_{x, y, \mu} + y \theta_{x, y, \lambda} = 2 \pi
\end{cases}
\]

As we mentioned above, 
a small deformation $\rho_{x, y}$ of $\rho$ corresponds to 
$(x, y)$ in an open neighborhood $U_{\infty} \subset \mathbb{R}^2 \cup \{ \infty \} = S^2$. 
Thurston's hyperbolic Dehn surgery theorem says that there is an integer $N > 0$ 
such that if relatively prime integers $p$ and $q$ satisfies $|p| + |q| > N$, 
then $\rho_{p, q}$ satisfies 
\[
p \log \alpha_{p, q, \mu} + q \log \alpha_{p, q, \lambda} = 2\pi i, 
\]
and gives an incomplete hyperbolic metric of $S^3 - K$ so that 
its completion is $K(p/q)$ in which $K_{p/q}^*$, the image of the filled solid torus, 
is the shortest closed geodesic in the hyperbolic $3$--manifold $K(p/q)$. 

Although $\rho$ is a faithful representation,  
$\rho_{p, q}$ is a representation of $\pi_1(S^3 - K)$ which is not faithful and 
its image $\Gamma_{p, q}$ is regarded as $\pi_1(K(p/q))$. 
\[
\xymatrix{
\pi_1(E(K))\ar@{->>}[drr]_{\rho_{p,q}}  \ar[rr]^{\rho}_{\cong} &  & \hspace{0.2cm} \Gamma\subset \mathrm{Isom}^+(\mathbf{H}^3)\ar@{->>}[d] & \\
              &  &  \hspace{1.0cm}  \Gamma_{p,q}\subset \mathrm{Isom}^+(\mathbf{H}^3)
}
\]

\begin{proposition}
\label{L>0}
Let $g$ be a non-trivial element of $G(K)$. 
Then 
$L(g) > 0$ if and only if $g$ is a non-peripheral element, 
i.e. it is not conjugate into $P(K) \subset G(K)$. 
\end{proposition}

\begin{proof}
Assume first  that $g$ is non-peripheral. 
Then $\rho(g)$ is loxodromic, 
and up to conjugation we may assume that 
its restriction to $\partial \mathbb{H}^3 = \mathbb{C} \cup \{ \infty \}$ is given by 
$\rho(g) \colon z \mapsto \alpha_g z$, 
where $\alpha_g = r_g e^{i \theta_g}$ and $r_g > 1$. 
Thus the translation length of $\rho(g)$ is $\left| \log |\alpha_g| \right| = | \log r_g | > 0$. 
Under small deformation of $\rho$, 
the translation length $| \log r_{x, y, g} |$ of $\rho_{x, y}(g)$ deforms continuously together with $(x, y) \in U_{\infty}$. 
Thus $g$ is freely homotopic to a unique closed geodesic $c_g \subset S^3 - K$, the image of the axis of 
$\rho(g) \subset \mathbb{H}^3$, 
of length $\ell_g =  | \log r_g | > 0$. 
Hence, we can take constants $\varepsilon_0 > 0$ and $N_0 > 0$ so that 
if $|p| + |q| > N_0$, 
then the length of $c_g$ in $K(p/q)$, 
i.e. $\ell_{K(p/q)}(g)$,  
is greater than $\varepsilon_0$, 
while the length of $K^*_{p/q} < \varepsilon_0$. 
Since there are only finitely many slopes $p_1/q_1, \dots, p_k/q_k$ with $|p_i| + |q_i| \le N_0$, 
let us put $\varepsilon_1 = \mathrm{min}\{ \ell_{K(p_1/q_1)}(p_{p_1/q_1}(g)), \dots, \ell_{K(p_k/q_k)}(p_{p_k/q_k}(g)) \}$, 
where $K(p_i/q_i)$ is hyperbolic and $p_{p_i/q_i}(g) \ne 1$ in $\pi_1(K(p_i/q_i))$ ($i = 1, \dots, k$). 
Finally put $\varepsilon = \mathrm{min}\{ \varepsilon_0, \varepsilon_1 \}$. 
Then the length of $c_g$ in $K(p/q)$, i.e. $\ell_{K(p/q)}(g)$ is greater than or equal to $\varepsilon > 0$ for all 
hyperbolic surgery slopes $p/q \in \mathbb{Q}$ with $p_{p/q}(g) \ne 1$.  
This shows that $L(g) \ge \varepsilon > 0$.

Suppose that $g$ is peripheral. 
Then it is conjugate to $(\mu^a \lambda^b)^m \in P(K)$ for some integers $a, b$ and $m$, 
where $a$ and $b$ are relatively prime. 
Thus $p_{a/b}(g) = 1$; 
otherwise $p_{p/q}(g)$ ($p/q \ne a/b$) is freely homotopic to $n_{p/q}$--th power of the dual knot $K^*_{p/q}$, 
the core of the filling solid torus in $K(p/q)$. 
It should be noted that $n_{p/q}$ varies according as $p/q$. 
So to controll $n_{p/q}$, 
we take an infinite sequence of surgery slopes 
$p_i/q_i$ so that the distance $\Delta(a/b,\ p_i/q_i) = |aq_i - bp_i| = 1$ with $|p_i| + |q_i| \to \infty$ $(i \to \infty)$. 
Then a representative of the slope element $\mu^a \lambda^b$ wraps once in the $p_i/q_i$--filling solid torus, so it is freely homotopic to $K^*_{p_i/q_i}$, i.e. $n_{p_i/q_i} = 1$.  
This shows that $\ell_{K(p_i/q_i)}(p_{p_i/q_i}(g)) = m \ell(K^*_{p_i/q_i})$. 

Thurston's hyperbolic Dehn surgery theorem says that 
when $|p_i| + |q_i|$ goes to $\infty$, 
$K(p_i/q_i)$ converges to $S^3 - K$ and $K^*_{p_i/q_i}$ is a closed geodesic whose length tends to $0$. 
Hence, when $|p_i| + |q_i|$ goes to $\infty$, 
$\ell_{K(p_i/q_i)}(p_{p_i/q_i}(g)) = m \ell(K^*_{p_i/q_i})$ also tends to $0$, and so $L(g) = 0$. 
\end{proof}

\bigskip

\subsection{Stable commutator length vs. hyperbolic length}
In this subsection we will recall 
the definition of the stable commutator length, and then collect some useful results. 

For $g \in [G,G]$ the \emph{commutator length\/} $\mathrm{cl}_G(g)$ is the smallest number of commutators in $G$ whose product is equal to $g$. 
The \emph{stable commutator length\/} $\mathrm{scl}_{G}(g)$ of $g \in [G,G]$ is defined to be the limit
\begin{equation}
\label{def1}
\mathrm{scl}_{G}(g) = \lim_{n \to \infty} \frac{\mathrm{cl}(g^n)}{n}.
\end{equation}

Since $\mathrm{cl}_G(g^n)$ is non-negative and subadditive, Fekete's subadditivity lemma shows that 
this limit exists.  

We will extend (\ref{def1}) to the stable commutator length $\mathrm{scl}(g)$ for an element $g$ which is not necessarily in $[G, G]$ as 
\begin{equation}
\label{def2}
\mathrm{scl}_G(g) = \begin{cases}
\frac{\textrm{scl}_G(g^{k})}{k} & \mbox{ if } g^{k} \in [G,G] \mbox{ for some } k > 0,\\
\infty & \mbox{otherwise}.
\end{cases}
\end{equation}

By definition (\ref{def1}), 
it is easy to observe that if $g^k,\, g^{\ell} \in [G, G]$, 
then $\displaystyle \frac{\mathrm{scl}(g^k)}{k} = \frac{\mathrm{scl}_G(g^{\ell})}{\ell}$ for any $k,\, \ell > 0$. 
So in (\ref{def2}) $\mathrm{scl}_G(g)$ is independent of the choice of $k > 0$ such that $g^k \in [G, G]$. 
In particular, we have the following property; see \cite[Lemma~2.1]{IMT_decomposition} for its proof. 

\begin{lemma}
\label{scl_g^k}
For any $g \in G$ and $k > 0$,  
we have $\mathrm{scl}_G(g^{k})=k\,\mathrm{scl}_G(g)$. 
\end{lemma}

Since a homomorphism sends a commutator to a commutator, 
we observe the following monotonicity property of scl \cite[Lemma~2.4]{Cal_MSJ}. 

\begin{lemma}
\label{monotonicity}
Let $\varphi\colon G \to H$ be a homomorphism. 
Then \[
\mathrm{scl}_H(\varphi(g)) \leq \mathrm{scl}_G(g)\quad  \textrm{for all}\quad g \in G.
\] 
\end{lemma}

Recall that a map $\phi\colon G \rightarrow \mathbb{R}$ is \emph{homogeneous quasimorphism\/} of defect $D(\phi)$ if
\[ D(\phi)=\sup_{g,h \in G}|\phi(gh)-\phi(g)-\phi(h)| < \infty, \quad \phi(g^{k})=k\phi(g) \ (\forall g\in G, k \in \mathbb{Z}).\]

Note that homogeneous quasimorphism is constant on conjugacy classes \cite[2.2.3]{Cal_MSJ}. 

\begin{lemma}
\label{quasimorphism_class_function}
A homogeneous quasimorphism $\phi \colon G \to \mathbb{R}$ satisfies 
$\phi(g^{-1}hg)=\phi(h)$ for all $g,h \in G$. 
\end{lemma}

To estimate the scl, the following Bavard's Duality Theorem is useful.

\begin{theorem}[Bavard's Duality Theorem \cite{Bavard}]
\label{Bavard}
For $g \in [G, G]$,
\[ \mathrm{scl}_{G}(g)=\sup_{\phi} \frac{|\phi(g)|}{2D(\phi)} \]
where $\phi\colon G \rightarrow \mathbb{R}$ runs all homogeneous quasimorphisms of $G$ which are not homomorphisms. 
\end{theorem}

The following inequality is well-known for the expert. 
Although this is directly proved by noting that $(gh)^{2n}g^{-2n}h^{-2n}$ can be written as a product of $n$ commutators \cite[p. 45]{Cal_MSJ},  
here we prove this using Bavard's duality. 

\begin{lemma}
\label{scl_product}
Let $G$ be a group, and let $g, h$ be non-trivial elements of $G$. 
Assume that abelianization $G/[G, G]$ is a finite group, 
or $g, h \in [G, G]$ \(hence $gh \in [G, G]$ as well\). 
Then we have 
\[
\mathrm{scl}_G(gh) \ge \mathrm{scl}_G(g) -\mathrm{scl}_G(h) -\frac{1}{2}.
\]
\end{lemma}
\begin{proof}
Put $g=ab$, $h=b^{-1}$; 
if $g, h \in [G, G]$, then $a, b, ab \in [G, G]$. 
The assertion is equivalent to the assertion  
\[
\mathrm{scl}_G(ab) \le  \mathrm{scl}_G(a) + \mathrm{scl}_G(b) +\frac{1}{2}. 
\]
Even when $a, b, ab \not\in [G, G]$, 
since $G/[G, G]$ is finite, 
we may assume that $a^m, b^n, (ab)^{\ell} \in [G, G]$ for some integers $m, n, \ell > 0$. 
For every $\varepsilon>0$, 
Bavard's duality implies that there exists a homogeneous quasimorphism $\phi\colon G\rightarrow \mathbb{R}$ such that 
$\mathrm{scl}((ab)^{\ell}) -\varepsilon < \frac{\phi((ab)^{\ell})}{2D(\phi)}$. 
Then
\begin{align*}
\mathrm{scl}(ab)-\varepsilon  
& = \frac{\mathrm{scl}((ab)^{\ell})}{\ell} - \varepsilon
<  \frac{\mathrm{scl}((ab)^{\ell})}{\ell} - \frac{\varepsilon}{\ell} = \frac{1}{\ell}\left(\mathrm{scl}((ab)^{\ell}) -\varepsilon \right) \\
&
< \frac{1}{\ell} \frac{\phi((ab)^{\ell})}{2D(\phi)}
= \frac{\phi(ab)}{2D(\phi)} 
\leq \frac{\phi(a)}{2D(\phi)}+ \frac{\phi(b)}{2D(\phi)} + \frac{D(\phi)}{2D(\phi)}\\
& \leq  \frac{|\phi(a)|}{2D(\phi)}+ \frac{|\phi(b)|}{2D(\phi)} +\frac{1}{2} 
 = \frac{1}{m} \frac{|\phi(a^m)|}{2D(\phi)} + \frac{1}{n} \frac{|\phi(b^n)|}{2D(\phi)} +\frac{1}{2} \\
& \leq \frac{\mathrm{scl}(a^m)}{m} + \frac{\mathrm{scl}(b^n)}{n}  +\frac{1}{2} \\
& = \mathrm{scl}(a) + \mathrm{scl}(b) +\frac{1}{2}.
\end{align*}
\end{proof}

\medskip

Let $M$ be a hyperbolic $3$--manifold and $g$ a non-trivial element in $\pi_1(M)$. 
Recall that $\ell_M(g)$ denotes the length of a unique closed geodesic $c_g$ 
in a free homotopy class of $g$.
While hyperbolic length enjoys the property described in Lemma~\ref{scl_g^k}, 
it is not known to satisfy the monotonicity property described in Lemma~\ref{monotonicity} for 
the epimorphism $p_r \colon G(K) \to  \pi_1(K(r))$. 

\smallskip

The next strong result due to Calegari \cite{Cal_GAFA} relates
stable commutator length $\textrm{scl}_{\pi_1(M)}(g)$ and hyperbolic length $\ell_M(g)$. 

\begin{theorem}[Length inequality \cite{Cal_GAFA}]
\label{theorem:length-inequality}
For any $\varepsilon>0$ there exists a constant $\delta(\varepsilon)>0$ 
such that if $M$ is a complete hyperbolic $3$--manifold $M$ and 
$g \in \pi_1(M)$ with $\mathrm{scl}_{\pi_1(M)}(g) \le \delta(\varepsilon)$, 
then $\ell_{M}(g) \leq \varepsilon$.
\end{theorem}

In Theorem~\ref{theorem:length-inequality}, 
it should be emphasized that the constant $\delta(\varepsilon)$ is independent of $M$, it depends only on $\varepsilon > 0$. 
Thanks to this universality, 
as a consequence of Proposition~\ref{L>0} and Theorem~\ref{theorem:length-inequality}, 
we obtain the following universal lower bound $\delta_g > 0$ for stable commutator length of $p_s(g) \ne 1$ independent of slopes $s$. 

\begin{theorem}[Lower bound of scl over hyperbolic surgery slopes]
\label{scl_bound}
Let $K$ be a hyperbolic knot and $g \in G(K)$ a non-peripheral element. 
Then we have a constant $\delta_g > 0$ such that for all hyperbolic surgery slopes $s \in \mathbb{Q}$, 
\[
\mathrm{scl}_{\pi_1(K(s))}(p_s(g)) > \delta_g > 0\quad \mbox{whenever}\quad p_s(g) \neq 1.
\]
\end{theorem}

\begin{proof}
Since $g$ is a non-peripheral element, 
Proposition~\ref{L>0} shows that $L(g) > 0$. 
Put $\varepsilon = L(g)/2 > 0$. 
Then Theorem~\ref{theorem:length-inequality} shows that there exists a constant 
$\delta(\varepsilon)> 0$ such that 
if $\textrm{scl}_{\pi_1(K(s))}(p_s(g)) \le \delta(\varepsilon)$ for $p_s(g) \ne 1$ in $\pi_1(K(s))$, 
then 
\[
\ell_{K(s)}(p_s(g)) \le \varepsilon = L(g)/2 < L(g).
\] 
Suppose for a contradiction that there exists a hyperbolic surgery slope $s \in \mathbb{Q}$ such that $p_s(g) \ne 1$ in $\pi_1(K(s))$ and $\mathrm{scl}_{\pi_1(K(s))}(p_s(g)) \le \delta(\varepsilon)$. 
Then $\ell_{K(s)}(p_s(g)) \le \varepsilon  < L(g)$, contradicting the definition of $L(g)$. 
Hence, for any hyperbolic surgery slope $s$ with $p_s(g) \ne 1$, 
we have $\textrm{scl}_{\pi_1(K(s))}(p_s(g)) > \delta(\varepsilon) = \delta(L(g)/2)$. 
Let us put $\delta_g = \delta(\varepsilon) = \delta(L(g)/2)$ to obtain the desired bound. 
\end{proof}

\medskip

On the contrary, 
if $g$ is a peripheral element, 
i.e. it is conjugate to $\gamma^d$ for some integer $d > 0$, 
where $\gamma = \mu^a \lambda^b$ for some relatively prime integers $a$ and $b$,
then 
$p_{s}(g) = 1$ and hence $\mathrm{scl}_{\pi_1(K(s))}(p_s(g)) = 0$ for $s = a/b \in \mathbb{Q}$. 
Furthermore, we may prove the following.

\begin{proposition}
\label{scl_bound_peripheral}
Let $K$ be a hyperbolic knot and $g \in G(K)$ a non-trivial peripheral element.  
Then, for any $\varepsilon > 0$, there exists $s \in \mathbb{Q}$ such that 
\[
\mathrm{scl}_{\pi_1(K(s))}(p_s(g)) < \varepsilon\quad \textrm{and}\quad p_s(g) \ne 1 \in \pi_1(K(s)). 
\]
\end{proposition}

\begin{proof}
Taking a suitable conjugation and its inverse, we may assume $g = \gamma^d$ for some integer $d > 0$, 
where $\gamma = \mu^a \lambda^b$ for some relatively prime integers $a$ and $b \ge 0$. 
By the assumption $\gamma \ne \lambda$, $a \ne 0$. 

In the following we show 
$\mathrm{scl}_{\pi_1(K(s))}(p_s(\gamma)) < \varepsilon' = \varepsilon/d$, 
because
\[
\mathrm{scl}_{\pi_1(K(s))}(p_s(g)) = \mathrm{scl}_{\pi_1(K(s))}(p_s(\gamma^d)) = d \cdot \mathrm{scl}_{\pi_1(K(s))}(p_s(\gamma)).
\] 
Let us consider $\gamma^{n} \lambda^{-1} = \mu^{an} \lambda^{bn -1} \in \pi_1(\partial(E(K))$. 

\begin{claim}
\label{coprime}
There are infinitely many integers $n$ such that $an$ and $bn - 1$ are relatively prime. 
\end{claim}

\begin{proof}
Without loss of generality, 
we may
assume $a > 0$. 
If $ab \le 1$, then $b = 0$ or $a = b =1$.  
In the former case, 
$a = 1$ and $\mathrm{GCD}(an, bn-1) = (n, -1) = 1$ for all integers $n$. 
In the latter case, $\mathrm{GCD}(an, bn-1) = (n, n-1) = 1$ for all integers $n$.
In the following we assume $ab>1$. 
Take $n= a^{2s}b^{2t-1}$. 
Then $an= a^{2s+1}b^{2t-1}$ and 
$bn-1 = a^{2s}b^{2t}-1= (a^sb^t -1)(a^s b^t +1)$. 
Since $a$ and $b$ are coprime, 
a prime number $p$ dividing $an$ divides also $a$ or $b$. 
However, it does not divide $bn-1$. 
\end{proof}

Let us put $s_n = \frac{an}{bn-1} \in \mathbb{Q}$. 
Then $p_{s_n}(\gamma^{n}\lambda^{-1}) = p_{s_n}(\mu^{an} \lambda^{bn-1}) =1$, 
namely $p_{s_n}(\gamma)^n = p_{s_n}(\lambda)$ in $\pi_1(K(s_n))$. 
Thus 
\[
n \cdot \mathrm{scl}_{\pi_1(K(s_n))}(p_{s_n}(\gamma)) = \mathrm{scl}_{\pi_1(K(s_n))}(p_{s_n}(\gamma)^n) = \mathrm{scl}_{\pi_1(K(s_n))}(p_{s_n}(\lambda)).
\]
Since $\textrm{scl}_{G(K)}(\lambda) = g(K)-\frac{1}{2} < g(K)$ (Remark~\ref{scl_longitude} below), we have
\[
 \mathrm{scl}_{\pi_1(K(s_n))}(p_{s_n}(\gamma)) = \frac{\mathrm{scl}_{\pi_1(K(s_n))}(p_{s_n}(\lambda))}{n} \le \frac{\mathrm{scl}_{G(K)}(\lambda)}{n} <  \frac{g(K)}{n} < \varepsilon'
 \]
 for $n > \frac{g(K)}{\varepsilon'}$. 
This then implies 
\[
\mathrm{scl}_{\pi_1(K(s_n))}(p_{s_n}(g)) < \varepsilon\quad \textrm{for}\ n > \frac{d \cdot g(K)}{\varepsilon}
\]
as mentioned above. 

Finally we observe that $p_{s_n}(g) \ne 1$ for all but finitely many $n$. 
If $p_{s_n}(g) = 1 \in \pi_1(K(s_n))$, 
then $g = \gamma^d \in  \langle\!\langle s_n \rangle\!\rangle$. 
By Proposition~\ref{slope} $s_n$ is represented by the slope element $\gamma$ or $s_n$ is a finite surgery slope. 
For any hyperbolic knot $K$ the number of such surgeries is finite. 

This completes a proof. 
\end{proof}

\begin{remark}
\label{scl_longitude}
We remark that for a non-trivial knot $K$ 
the stable commutator length of $\lambda$ is expressed by the knot genus as follows \cite[Proposition 4.4]{Cal_MSJ}\textup{:} 
\[
\mathrm{scl}_{G(K)}(\lambda)=g(K)-\frac{1}{2} \ge \frac{1}{2}.
\]
Hence, 
\[
\mathrm{scl}_{\pi_1(K(s))}(p_s(\lambda)) \le \mathrm{scl}_{G(K)}(\lambda) =g(K)-\frac{1}{2}.
\]
\end{remark}

\medskip

Even if $L(g) > 0$, 
Theorem~\ref{scl_bound} does not assert that 
$p_s(g) \ne 1$ for ``all'' hyperbolic surgery slopes. 
To find an element $g' \in G(K)$ such that $p_s(g') \ne 1$ for ``all'' hyperbolic surgery slopes, 
we need further procedure.

Using stable commutator length, 
we give an efficient way to construct elements in $G(K)$ which are not trivialized by all hyperbolic Dehn fillings.
See Lemma~\ref{non_vanish_non_finite}, 
whose proof requires the following result. 

For any non-peripheral element $g \in G(K)$, 
$\delta_g$ denotes a positive constant given by Theorem~\ref{scl_bound}. 

\begin{proposition}
\label{construction}
Let $K$ be a hyperbolic knot in $S^3$ and $x$ a non-trivial element in $G(K)$. 
Assume that 
$s$ is a hyperbolic surgery slope, 
and $p_s(x)$ is non-trivial in $\pi_1(K(s))$. 
\begin{enumerate} 
\renewcommand{\labelenumi}{(\arabic{enumi})}
\item If $x$ is non-peripheral \(i.e. $L(x) > 0$\), 
then for an element $y \in [G(K),G(K)]$ and $p > \frac{\mathrm{scl}_{G(K)}(y)}{\delta_x}$, 
$p_s(x^{-p}y) \neq 1$ for the hyperbolic surgery slope $s$. 
\item If $x \in [G(K),G(K)]$, 
then for a non-peripheral element $y \in G(K)$ \(i.e. $L(y) > 0$\) and $p > \frac{\mathrm{scl}_{G(K)}(x)}{\delta_y}$, 
$p_s(x^{-1}y^{p}) \neq 1$ for the hyperbolic surgery slope $s$. 
\end{enumerate}
\end{proposition}

\begin{proof}
$(1)$\  
Assume to the contrary that  $p_s(x^{-p}y) = 1$, i.e. $p_s(x^p) = p_s(y)$. 
Then we have 
\begin{align*}
\textrm{scl}_{G(K)}(y) & < p\, \delta_x\\
&< p\ \textrm{scl}_{\pi_1(K(s))}(p_s(x)) \quad (\mbox{Theorem~\ref{scl_bound}}) \\
& = \textrm{scl}_{\pi_1(K(s))}(p_s(x)^p) \quad (\mbox{Lemma~\ref{scl_g^k}}) \\
& = \textrm{scl}_{\pi_1(K(s))}(p_s(x^{p}))\\
& = \textrm{scl}_{\pi_1(K(s))}(p_s(y))\\
& \leq \textrm{scl}_{G(K)}(y)\quad (\mbox{Lemma~\ref{monotonicity}}). 
\end{align*}
This is a contradiction.

$(2)$\  
Assume for a contradiction that $p_s(x^{-1}y^p) = 1$, i.e. $p_s(x) = p_s(y^p)$. 
Since $p_s(x)$ is non-trivial, $p_s(y^p) = p_s(x) \ne 1$, in particular $p_s(y) \ne 1$. 
Thus we have
\begin{align*}
\textrm{scl}_{G(K)}(x) & <  p\, \delta_y \\
& < p \ \textrm{scl}_{\pi_1(K(s))}(p_s(y)) \quad (\mbox{Theorem~\ref{scl_bound}})\\
& = \textrm{scl}_{\pi_1(K(s))}(p_s(y)^p)\quad (\mbox{Lemma~\ref{scl_g^k}})\\
& = \textrm{scl}_{\pi_1(K(s))}(p_s(y^p))\\
& = \textrm{scl}_{\pi_1(K(s))}(p_s(x))\\
& \leq  \textrm{scl}_{G(K)}(x)\quad (\mbox{Lemma~\ref{monotonicity}}). 
\end{align*}
This is a contradiction. 
\end{proof}

\medskip

For hyperbolic knots any power of a non-peripheral element is also non-peripheral. 
(This is not true for torus knots. 
Actually, some power of a non-peripheral element represented by an exceptional fiber is represented by a regular fiber, and thus peripheral.)
For later convenience we generalize this to the following result. 

\begin{lemma}
\label{product_non-peripheral}
Let $K$ be a hyperbolic knot and $g, h$ be elements in $G(K)$. 
Assume that $g$ is non-peripheral. 
Then we have the following. 
\begin{enumerate}
\renewcommand{\labelenumi}{(\arabic{enumi})}
\item
$g^n h$ are non-peripheral except for only finitely many integers $n$. 
\item
If $h$ is peripheral, then there are infinitely many integers $n$ such that $g^n h$ are mutually non-conjugate elements in $G(K)$. 
\end{enumerate}
\end{lemma}

\begin{proof}
(1) Consider a holonomy representation $\rho\colon G(K) \to \mathrm{PSL}(2, \mathbb{C})$. 
We denote an element in
$\mathrm{PSL}(2, \mathbb{C}) = \mathrm{SL}(2, \mathbb{C})/ \{ \pm I \}$ 
by 
$\begin{bmatrix}
a & b \\
c & d
\end{bmatrix}$
to distinguish from 
a matrix 
$\begin{pmatrix}
a & b \\
c & d
\end{pmatrix}$ 
in $\mathrm{SL}(2, \mathbb{C})$.

Recall that the trace is invariant conjugation and $\gamma$ is peripheral if and only if $\mathrm{tr}(\rho(\gamma) )= \pm 2$. 
Since $g$ is non-peripheral, 
we may take the holonomy so that 
$\rho(g) =  \begin{bmatrix}
	\alpha & 0 \\[2pt]
	0 & \alpha^{-1}
\end{bmatrix}$ and 
$\rho(h) =  \begin{bmatrix}
	x & y \\[2pt]
	z & u
\end{bmatrix}$, 
where $|\alpha| > 1$. 
Then  $\mathrm{tr}(\rho(g^{n}h)) = \alpha^{n}x+ \alpha^{-n}u$.  
If $x = u = 0$, 
then $\mathrm{tr}(\rho(g^{n}h) ) =  0$ independent of $n$, 
and thus $g^{n} h$ are non-peripheral for any integer $n$. 
So we may assume $x \ne 0$ or $u \ne 0$. 
Then it is easy to see that $\alpha^{n}x+ \alpha^{-n}u$ can be $\pm 2$ for only finitely many integers $n$, 
and hence $g^{n}h$ are non-peripheral except for these finitely many integers.  

(2) By the assumption $\mathrm{tr}(\rho(h)) = \pm 2 \ne 0$, 
in particular $x \ne 0$ or $u \ne 0$. 
Then $\mathrm{tr}(\rho(g^{n}h)) = \alpha^{n}x+ \alpha^{-n}u \to \infty$ (if $x \ne 0$) or $0$ (if $x = 0$). 
Thus there are infinitely many integers $n$ such that $g^n h$ are mutually non-conjugate. 
\end{proof}

\begin{remark}
\label{commute}
In the proof of Lemma~\ref{product_non-peripheral} $\mathrm{tr}(\rho(h g^n)) = \mathrm{tr}(\rho(g^n h))$, 
so $(1)$ and $(2)$ hold for $h g^n$ as well.  
\end{remark}

\section{Elements $g$ with $\mathcal{S}_K(g) = \emptyset$} 
\label{homologically0}

By definition, an element $g \in G(K)$ satisfying $\mathcal{S}_K(g) = \emptyset$ survives for all non-trivial Dehn fillings, 
so it is strong against Dehn fillings.   

The simplest such an element is a meridian $\mu$ or its conjugate. 
It should be emphasized that this fact is already highly non-trivial, 
because it is equivalent to Property P Theorem \cite{KM}.  
Furthermore, any non-trivial powers of the meridian $\mu$ also satisfies $\mathcal{S}_K(\mu^n) = \emptyset$ ($n \ne 0$) 
as long as $K$ has no finite surgery slopes. 
See Proposition~\ref{slope}. 

On the other hand, since these elements are homologically non-trivial,
it is not clear whether or not we have a homologically trivial element $g\, (\in [G(K), G(K)])$ with $\mathcal{S}_K(g) = \emptyset$. 

\begin{example}[Figure-eight knot]
\label{fig-eight}
Let $K$ be the figure-eight knot with 
\[
G(K) = \langle a, b \mid ab^{-1}a^{-1}ba = bab^{-1}a^{-1}b \rangle,
\] 
where $a$ and $b$ are meridians as indicated by Figure~\ref{figure-eight_knot_a_b}. 
If $p_r([a,b])=1$ for some $r \in \mathbb{Q}$, 
then $\pi_1(K(r))$ must be cyclic.
This is impossible, because $K$ does not admit a cyclic surgery.
Thus the commutator $[a,b]$ satisfies $\mathcal{S}_K([a, b]) = \emptyset$. 

A similar argument applies whenever $G(K)$ is generated by two elements and $K$ has no cyclic surgeries. 
\end{example}

\begin{figure}[htb]
\centering
\includegraphics[width=0.2\textwidth]{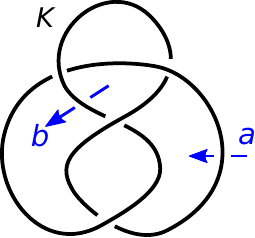}
\caption{$\mathcal{S}_K([a, b]) = \emptyset$ for the figure-eight knot $K$.} 
\label{figure-eight_knot_a_b}
\end{figure}

If $s$ is a cyclic surgery slope, but it is not a finite surgery slope, 
then $\pi_1(K(s)) \cong \mathbb{Z}$. 
However, this happens only when $K$ is the unknot and $s = 0$ \cite{GabaiIII}. 
Thus for non-trivial knots cyclic surgeries are finite surgeries. 

If $K$ has a cyclic surgery slope $s$, 
then $\pi_1(K(s)) \cong G(K)/ \langle\!\langle s \rangle\!\rangle$ is abelian, 
hence $[G(K),G(K)] \subset \langle\!\langle s \rangle\!\rangle$. 
Thus for every element $g$ in $[G(K),G(K)]$, 
$p_s(g) = 1$ in $\pi_1(K(s))$. 

The goal of this section is to prove Theorem~\ref{hyperbolic_persistent_[G,G]_noncyclic}. 

\begin{thm_hyperbolic_persistent_[G,G]_noncyclic}[Elements with $\mathcal{S}_K(g) = \emptyset$]
Let $K$ be a hyperbolic knot in $S^{3}$.  
Then there exist infinitely many, mutually non-conjugate elements $g \in [G(K),G(K)]$ 
such that $p_s(g) \neq 1$ in $\pi_1(K(s))$ for all non-cyclic surgery slopes $s \in \mathbb{Q}$. 
In particular, 
if $K$ has no cyclic surgery, then there exist infinitely many, mutually non-conjugate elements $g \in [G(K),G(K)]$ such that 
$\mathcal{S}_K(g) = \emptyset$. 
\end{thm_hyperbolic_persistent_[G,G]_noncyclic}

\begin{proof}
The longitude $\lambda$ of $K$ is the unique slope element which belongs to $[G(K), G(K)]$.  
As a particular case of Proposition~\ref{slope}, 
we have 
\medskip

\begin{lemma}
\label{p(lambda^q)=1}
If $p_s(\lambda^q) = 1$ for some $q \ne 0$, 
then $s=0$ or $s$ is a finite surgery slope. 
\end{lemma}

\begin{proof}
The assumption $p_s(\lambda^q) = 1$ implies that 
$\lambda^q \in \langle\!\langle s \rangle\!\rangle$. 
By Proposition~\ref{slope} $s = 0$ or $s$ is a finite surgery slope. 
\end{proof}

Let $\Sigma$ be a minimal genus Seifert surface of a hyperbolic knot $K$.  
Take a non-trivial element $y \in \pi_1(\Sigma)$ so that $[y] \ne 0$ in $H_1(\Sigma)$, 
equivalently $y \not\in [\pi_1(\Sigma), \pi_1(\Sigma)]$. 
(In particular, $y$ is not conjugate into $\pi_1(\partial \Sigma)$ in $\pi_1(\Sigma)$.)
Let $i\colon \Sigma \rightarrow E(K)$ be the inclusion map. 
Since $\Sigma$ is incompressible, 
$i_* \colon \pi_1(\Sigma) \to G(K)$ is injective. 
So we may regard $\pi_1(\Sigma) \subset G(K)$ and 
use the same symbol $y$ to denote $i_*(y) \in \pi_1(\Sigma) \subset G(K)$. 

\begin{claim}
\label{non-conjugate_non-peripheral}
$y$ is non-peripheral in $G(K)$, 
i.e. 
it is not conjugate into $\pi_1(\partial E(K))$ in $G(K)$. 
\end{claim}

\begin{proof}
Assume for a contradiction that $y$ is conjugate to $y' \in \pi_1(\partial E(K))$ in $G(K)$. 
Since $y \in [G(K), G(K)]$, so is $y'$, 
and we may assume a representative of $y'$ lies in $\partial \Sigma$. 
Thus we write $y' = \lambda^{\ell}$ for some integer $\ell \ne 0$. 
So $y$ is conjugate to $\lambda^{\ell}$. 

Let $\widehat{\Sigma}$ be the non-separating closed surface in $K(0)$ obtained from $\Sigma$ by capping off 
$\partial \Sigma$ by a meridian disk of the filling solid torus; $\mathrm{genus}(\Sigma) = \mathrm{genus}(\widehat{\Sigma})$. 
Then the inclusion $\Sigma \to \widehat{\Sigma}$ induces an isomorphism $H_1(\Sigma) \to H_1(\widehat{\Sigma})$.  
Since $[y] \ne 0 \in H_1(\Sigma)$, we also have $[y] \ne 0$ in $H_1(\widehat{\Sigma})$.  
Hence, $y$ remains non-trivial in $\pi_1(\widehat{\Sigma})$. 
Assume for a contradiction that $\widehat{\Sigma}$ is compressible. 
Then apply a compression to $\widehat{\Sigma}$ to obtain a non-separating surface in $K(0)$ whose genus is strictly smaller than that of $\widehat{\Sigma}$. 
This contradicts  \cite[Corollary~8.3]{GabaiIII}. 
Hence, $\widehat{\Sigma}$ is incompressible in $K(0)$,  
in particular, $y$ is non-trivial in $\pi_1(K(0))$. 

On the other hand, since $y$ is conjugate to $\lambda^{\ell}$ in $G(K)$, 
$y$ would be trivial in $\pi_1(K(0))$, a contradiction. 
Therefore $y$ is not conjugate into $\pi_1(\partial E(K))$ in $G(K)$.   
\end{proof}

\begin{lemma}
\label{non_finite}
Let $s$ be a non-finite surgery slope and $q$ a positive integer. 
Assume that $p_s(\lambda^q) = p_s(y)^p$ for some integer $p > 0$. 
Then we have\textup{:} 
\begin{enumerate}
\renewcommand{\labelenumi}{(\arabic{enumi})}
\item
$p_s(\lambda^q) \ne 1$, and 
\item 
$p$ is the unique integer which satisfies $p_s(\lambda^q) = p_s(y)^p$.  
We denote this integer $p$ by $p_{q, s}$. 
\end{enumerate}
\end{lemma}

\begin{proof}
(1)\ Suppose for a contradiction that $p_s(\lambda^q) = 1$. 
Then, since $s$ is not a finite surgery slope, 
Lemma~\ref{p(lambda^q)=1} shows that $s = 0$. 

Let $\widehat{\Sigma}$ be a non-separating closed surface in $K(0)$ 
obtained from $\Sigma$ by capping off $\partial \Sigma$ by a meridian disk of the filling solid torus as in the proof of Claim~\ref{non-conjugate_non-peripheral}.
The choice of $y$, $[y] \ne 0 \in H_1(\Sigma)$, 
assures that $[y]$ and $[y^p]$ are also non-trivial in $H_1(\widehat{\Sigma})$, 
and hence $y$ and $y^p$ are non-trivial in $\pi_1(\widehat{\Sigma})$.  
Since $\widehat{\Sigma}$ is incompressible in $K(0)$ as observed in the proof of Claim~\ref{non-conjugate_non-peripheral}, 
$p_0(y^p) \ne 1$ in $\pi_1(K(0))$. 
So we have $1 = p_0(\lambda) = p_0(\lambda^q) = p_0(y)^p = p_0(y^p) \ne 1$ in $\pi_1(K(0))$, 
a contradiction. 

\smallskip

(2)\ 
Suppose for a contradiction that $p_s(y)^{p'} = p_s(\lambda^q)  = p_s(y)^{p}$ for $p' \ne p$. 
Then $p_{s}(\lambda^q)^{p'} = (p_s(y)^p)^{p'} = (p_s(y)^{p'})^p = p_s(\lambda^q)^p$, 
i.e. $p_s(\lambda^{(p' - p)q}) = 1$ ($(p' - p)q \ne 0$).  
So by Lemma~\ref{p(lambda^q)=1} the non-finite surgery slope $s$ must be $0$. 
Also we have $p_0(y)^{p' - p} = 1$. 
As shown in (1) $p_{0}(y) \ne 1$, 
so $|p' - p| \ne 1$, and hence $|p' - p| \ge 2$. 
Thus $p_{0}(y)$ is a torsion element in $\pi_1(K(0))$. 
However since $K(0)$ is irreducible \cite{GabaiIII} and $|\pi_1(K(0))| = \infty$, 
$\pi_1(K(0))$ has no torsion. 
This is a contradiction. 
\end{proof}

Following Thurston's hyperbolic Dehn surgery theorem 
there are at most finitely many non-hyperbolic surgery slopes, 
in particular, 
there are at most finitely many non-hyperbolic, non-finite surgery slopes $s_1, \dots, s_m$. 
For each such surgery slope $s_i$ ($1 \le i \le m$), 
at most one integer $p_{q, s_i} \ge 1$ satisfies $p_{s_i}(\lambda^q)=p_{s_i}(y)^{p_{q, s_i}}$ 
by Lemma~\ref{non_finite}(2). 

By Claim~\ref{non-conjugate_non-peripheral} and Proposition~\ref{L>0}, we see that $L(y) > 0$ for the non-peripheral element $y \in \pi_1(\Sigma) \subset G(K)$. 
Then for $\delta_y > 0$ given in Theorem~\ref{scl_bound}, 
we may take $p$ so that $p \delta_y > \textrm{scl}_{G(K)}(\lambda^q)$, 
equivalently $p > \dfrac{\textrm{scl}_{G(K)}(\lambda^q)}{\delta_y}$. 
The right hand side is a constant (once we fix a positive integer $q$), 
which we denote by $C_{q} > 0$.

\begin{lemma}
\label{non_vanish_non_finite}
For any non-finite surgery slope $s$ we have 
$p_s(\lambda^{-q} y^{p}) \ne 1$ if $p > C_{q}$ and $p \ne p_{q, s_1}, \dots, p_{q, s_m}$. 
\end{lemma}

\begin{proof}
Assume for a contradiction that $p_s(\lambda^{-q} y^{p})=1$, 
i.e. $p_s(\lambda^q) = p_s(y^p) = p_s(y)^p$ for some non-finite surgery slope $s$. 

\smallskip

Since $p_s(\lambda^q) \ne 1$ (Lemma~\ref{non_finite}), 
which also implies $p_s(y)^p \ne 1$, in particular $p_s(y) \ne 1$. 
Let us divide our arguments into two cases according as $s$ is a hyperbolic surgery slope or not. 
(Note that $0$--slope is a non-finite surgery slope for homological reason, and obviously 
$p_0(\lambda^q) = 1$. 
So Lemma~\ref{non_finite} shows that $s \ne 0$, i.e. $p_0(\lambda^{-q} y^{p}) \ne 1$.)

\medskip

\noindent
\textbf{Case 1.}\ $s$ is a hyperbolic surgery slope. 

Since $p_s(y) \ne 1$,  
Theorem~\ref{scl_bound} shows that $\textrm{scl}_{\pi_1(K(s))}(p_s(y)) > \delta_y > 0$. 
Following the assumption $p > \dfrac{\textrm{scl}_{G(K)}(\lambda^q)}{\delta_y}$, 
we may apply Proposition~\ref{construction}(2) with 
$x = \lambda^q$ to see that $p_s(\lambda^{-q}y^p) \ne 1$.  
This contradicts the initial assumption of the proof of Lemma~\ref{non_vanish_non_finite}. 

\medskip

\noindent
\textbf{Case 2.}\ 
$s$ is a non-hyperbolic surgery slope. 

Then $s = s_i$ for some $1 \le i \le m$, and
$p_{s_i}(\lambda^q) \ne p_{s_i}(y)^p$ since $p \neq p_{q, s}$ (Lemma~\ref{non_finite}).  
Thus $p_s(\lambda^{-q}y^p) \ne 1$. 
This contradicts the initial assumption of the proof of Lemma~\ref{non_vanish_non_finite}. 
\end{proof}

\medskip

Let us turn to the case where $s$ is a finite surgery slope.  
To consider finite surgeries we need to pay attention for the choice of 
$y \in \pi_1(\Sigma)$ and a positive integer $q$. 
(Accordingly we will re-take $y \in \pi_1(\Sigma) \subset G(K)$ and the positive integer $q$ in the above as well.)

Let $\mathcal{F} = \{ f_1, \dots, f_n\} \subset \mathbb{Q}$ be the set of finite surgery slopes which are not cyclic surgery slopes.

\begin{claim}
\label{number_finite}
$n \le 2$.
\end{claim}

\begin{proof}
Ni and Zhang \cite{Ni-Zhang-Finite} show that 
a hyperbolic knot in $S^{3}$ has at most three finite surgeries, and thus $n \le 3$; 
furthermore, if $n = 3$, then $K$ is the pretzel knot $P(-2, 3, 7)$. 
Recall that $P(-2, 3, 7)$ has two cyclic surgeries and one non-cyclic finite surgery. 
Hence $n \le 2$. 
\end{proof}

\medskip 

It is known that any knot admitting a non-trivial finite surgery must be a fibered knot,
and hence if $\mathcal{F} \ne \emptyset$, 
then $K$ is a fibered knot.  
See \cite{Ni} (cf. \cite{Ghi,Juh}).

Write $n_{f_i} = |\pi_1(K(f_i))|$ for each finite surgery slope $f_i$, 
and let $q_0$ be the least common multiples of $n_{f_1}, \dots, n_{f_n}$. 
Then for every element $g \in G(K)$,  
we have $p_{f_i}(g)^{q_0} = 1$ for $1 \le i \le n$. 

We begin by observing the following. 

\begin{lemma}
\label{non-peripheral}
There exists an element $y \in \pi_1(\Sigma) \subset G(K)$ with 
$0 \neq [y] \in H_1(\Sigma)$ which satisfies 
$p_{f_i}(y) \ne 1$ in $\pi_1(K(f_i))$ for any finite \(non-cyclic\) surgery slope $f_i$.  
\end{lemma}

\begin{proof}
Since $K$ is fibered, $\pi_1(\Sigma)=[G(K),G(K)]$. 
For each non-cyclic (non-abelian) surgery slope $f_i\in \mathcal{F}$, 
we have $[G(K),G(K)] - \langle\!\langle f_i \rangle\!\rangle 
= \pi_1(\Sigma) - \langle\!\langle f_i \rangle\!\rangle \ne \emptyset$. 
In fact, if $[G(K),G(K)] - \langle\!\langle f_i \rangle\!\rangle = \emptyset$, 
then $ \langle\!\langle f_i \rangle\!\rangle \supset [G(K),G(K)]$, and hence 
$\pi_1(K(f_i))$ is abelian, in particular cyclic. 
This contradicts $f_i$ being a non-cyclic surgery slope. 

Take a non-trivial element $y_i \in \pi_1(\Sigma) - \langle\!\langle f_i \rangle\!\rangle \ne \emptyset$. 
Then $p_{f_i}(y_i) \ne 1$.  
If $\mathcal{F} = \{ f_1 \}$, then put $y = y_1$. 
If $\mathcal{F} = \{ f_1, f_2 \}$, 
Then we put $y$ as follows;
\[ y= \begin{cases}
 y_1 & \mbox{ if } p_{f_2}(y_1)\neq 1, \\
 y_2 & \mbox{ if } p_{f_1}(y_2)\neq 1, \\
y_1 y_2 & \mbox{ if } p_{f_1}(y_2)=p_{f_2}(y_1)=1.
\end{cases}\]

Thus we have a non-trivial element $y \in \pi_1(\Sigma)$ with $p_{f_i}(y) \ne 1$ for $f_i \in \mathcal{F}$. 
Following Claim~\ref{non-conjugate_non-peripheral}, if $y$ is not conjugate into $\pi_1(\partial \Sigma)$ in $\pi_1(\Sigma)$, 
then it is non-peripheral in $G(K)$, and $y$ is a desired element. 

Suppose that $y \in \pi_1(\Sigma)$ is homologically trivial in $H_1(\Sigma)$, i.e. $[y]=0 \in H_1(\Sigma)$. 
Then $y \in [\pi_1(\Sigma), \pi_1(\Sigma)]$.
We replace $y$ with an element $y' \in \pi_1(\Sigma)$ 
which satisfies $[y] \ne 0$ in $H_1(\Sigma)$ and $p_{f_i}(y') \ne 1$ as follows. 
Take $x \not\in [\pi_1(\Sigma), \pi_1(\Sigma)]$ arbitrarily. 
Then by the choice of $q_0$ we have $p_{f_i}(x^{q_0}) = p_{f_i}(x)^{q_0} = 1$, 
and hence $p_{f_i}(x^{q_0}y) = p_{f_i}(x^{q_0}) p_{f_i}(y) = p_{f_i}(y) \ne 1$. 
It remains to show that 
$[x^{q_0}y] \ne 0$ in $H_1(\Sigma)$.  
Assume for a contradiction that $[x^{q_0}y] = 0$ in $H_1(\Sigma)$. 
Then $q_0 [x] + [y] = 0$ in $H_1(\Sigma)$. 
So $q_0 [x] = -[y] = 0$ in $H_1(\Sigma)$, 
which also implies $[x] = 0 \in H_1(\Sigma)$. 
Thus $x \in [\pi_1(\Sigma), \pi_1(\Sigma)]$, contradicting the choice of $x$. 
Hence, $[x^{q_0}y] \ne 0$ in $H_1(\Sigma)$, 
and Claim~\ref{non-conjugate_non-peripheral} shows that $y' = x^{q_0}y$ is non-peripheral in $G(K)$. 
Hence $y'=x^{q_0}y$ gives the desired element. 
\end{proof}

\medskip

Let us take $y \in \pi_1(\Sigma)$ as in Lemma~\ref{non-peripheral}, and put 
\[
\mathcal{P} = \{ p \mid p > C_{q_0}, \ 
p  \ne p_{q_0, s_1}, \dots, p_{q_0, s_m}\ \textrm{and}\ p \equiv  1\ \textrm{mod}\ q_0 \}. 
\]

Obviously $\mathcal{P}$ contains infinitely many integers $p$. 

We prove that $p_s(\lambda^{-q_0} y^{p}) \ne 1$ for all non-cyclic surgery slopes $s \in \mathbb{Q}$ and $p\in \mathcal{P}$.  
If $s$ is not a finite surgery slope, 
then Lemma~\ref{non_vanish_non_finite} assures that $p_s(\lambda^{-q_0} y^{p}) \ne 1$. 

Suppose that $s$ is a non-cyclic, finite surgery slope $f_i$. 
By definition of $q_0$, $q_0$ is a multiple of $n_{f_i}$ for $f_i \in \mathcal{F}$. 
So we have
\[
p_{f_i}(\lambda)^{q_0} = 1 \quad
\mathrm{and}\quad  
p_{f_i}(y)^{q_0} = 1.
\] 

Thus 
\[
p_{f_i}(\lambda^{-q_0} y^{p}) 
= p_{f_i}(\lambda)^{-q_0} p_{f_i}(y)^{tq_0 + 1} 
= p_{f_i}(y)^{t q_0 + 1}
= p_{f_i}(y) 
\ne 1.
\] 
 
\medskip

Finally, we show that there are infinitely many non-peripheral such elements up to conjugation. 

\begin{lemma}
\label{infinite_conjugacy}
There are infinitely many integers $p_i \in \mathcal{P}$ such that 
$\lambda^{-q_0}y^{p_i}$ are non-peripheral and mutually non-conjugate elements in $G(K)$. 
\end{lemma}

\begin{proof}
Recall that $\lambda$ is peripheral and $y$ is non-peripheral. 
Then Lemma~\ref{product_non-peripheral} (1) and Remark~\ref{commute} show that 
$\lambda^{-q_0}y^{p_i}$ are non-peripheral all but finitely many $p_i$.  
Furthermore, following Lemma~\ref{product_non-peripheral} (2) and Remark~\ref{commute}
we see that $\lambda^{-q_0}y^{p_i}$ are mutually non-conjugate for infinitely many integers $p_i$. 
\end{proof}

This completes a proof of Theorem~\ref{hyperbolic_persistent_[G,G]_noncyclic}.  
\end{proof}

\medskip
\section{Separation Lemma}
\label{finite_separation}

The goal of this section is to prove ``Separation Lemma'' for the function $\mathcal{S}_K \colon G(K) \to 2^{\mathbb{Q}}$. 

\begin{lemma_separation}[Separation Lemma]
Let $K$ be a hyperbolic knot.    
Let $\mathcal{R} = \{ r_1, \dots, r_n \}$ and $\mathcal{S} = \{ s_1, \dots, s_m \}$ be any finite, non-empty subsets of $\mathbb{Q}$ such that 
$\mathcal{R} \cap \mathcal{S} = \emptyset$. 
Assume that $\mathcal{S}$ does not contain a Seifert surgery slope.  
Then there exists an element $g \in [G(K), G(K)] \subset G(K)$ 
such that $\mathcal{R} \subset \mathcal{S}_K(g) \subset \mathbb{Q} - \mathcal{S}$. 
\end{lemma_separation}

\begin{proof}
We first observe: 

\begin{claim}
\label{at_most_two_torsion_surgeries}
$\mathcal{S}$ contains at most two torsion surgery slopes.
\end{claim}

\begin{proof}
Since $\mathcal{S}$ has no Seifert surgery slopes, torsion surgery slopes in $\mathcal{S}$ must be reducible surgery slopes.
It is known by \cite{GLreducible} that every knot has at most two reducing surgery slopes.
\end{proof}

In the following, without loss of generality, by re-indexing slopes in $\mathcal{S}$, 
we may assume the following. 

\medskip

$\bullet$ If $\mathcal{S}$ has a single torsion surgery slope, then $s_1$ is such a slope. 
If $\mathcal{S}$ contains exactly two torsion surgery slopes, 
then $s_1$ and $s_2$ are such slopes. 

In particular, we assume that $s_i$ $(i \ge 3)$ are not torsion surgery slopes. 

\medskip

For convenience, 
we sketch the construction of a desired element $g$ which satisfies  
$\mathcal{R} \subset \mathcal{S}_K(g) \subset \mathbb{Q} - \mathcal{S}$.  
We first find an element $g_{0,\, m}$ with 
$\mathcal{S}_K(g_{0,\, m}) \subset \mathbb{Q} - \mathcal{S}$, where $\mathcal{S}_K(g_{0,\, m})$ would be empty (Step 1). 
Then assuming we have an element $g_{k,\, m}$ with 
$\{ r_1, \dots, r_k \} \subset \mathcal{S}_K(g_{k,\, m}) \subset \mathbb{Q} - \mathcal{S}$, 
we will inductively construct $g_{k+1,\, m}$ 
satisfying $\{ r_1, \dots, r_k, r_{k+1} \} \subset \mathcal{S}_K(g_{k+1,\, m}) \subset \mathbb{Q} - \mathcal{S}$ (Step 2), 
and inductively we construct an element $g = g_{n,\, m}$ with the desired property 
$\{ r_1, \dots, r_n \} \subset \mathcal{S}_K(g_{n,\, m}) \subset \mathbb{Q} - \mathcal{S}$. 
The second step requires technical step-by-step modifications of $g_{k,\, m}$.   

\medskip

In the following, 
we write $\mathcal{R}_k = \{ r_1, \dots, r_k \} \subset \{ r_1, \dots, r_n \} = \mathcal{R}$. 
When $k = 0$, $\mathcal{R}_0$ is understood to be the emptyset. 

\medskip

\noindent
\textbf{Step 1.}\ 
The Geometrization theorem \cite{Pe1,Pe2,Pe3} says that a closed $3$--manifold with cyclic fundamental group is a lens space, 
which is also a Seifert fiber space. 
So by the assumption $\mathcal{S}$ does not contain a cyclic surgery slope. 

Let us apply Theorem~\ref{hyperbolic_persistent_[G,G]_noncyclic} to find an element $g_{0,\, m} \in [G(K),G(K)]$ so that 
$p_r(g_{0,\, m}) \ne 1$ whenever $r$ is not a cyclic surgery slope. 
Then obviously $g_{0,\, m}$ satisfies 
\[
\emptyset = \mathcal{R}_0 \subset \mathcal{S}_K(g_{0,\, m}) \subset \mathbb{Q} - \mathcal{S}. 
\]

\medskip

\noindent
\textbf{Step 2.}\ 
Assume that there exists an element $g_{k,\, m} \in [G(K),G(K)]$ such that 
\[
\mathcal{R}_k  = \{r_1, \ldots, r_{k}\} \subset \mathcal{S}_K(g_{k,\, m}) \subset \mathbb{Q} - \mathcal{S},
\]
\[
\mathrm{namely,}\ p_{r_i}(g_{k,\, m}) = 1\ (1 \le i \le k)\quad \textrm{and}\quad   p_{s_i}(g_{k,\, m}) \ne 1\ (1 \le i \le m). 
\]
\medskip

Then we will construct an element $g_{k+1,\, m} \in [G(K), G(K)]$ such that 
\[
\mathcal{R}_{k+1}  = \{r_1,\ldots,r_{k}, r_{k+1} \} \subset \mathcal{S}_K(g_{k+1,\, m}) \subset \mathbb{Q} - \mathcal{S}.
\]

\medskip

A proof of the second step requires modifications of $g_{k,\, m}$ several times to get $g_{k+1,\, m}$. 
First we observe the following useful fact.  

\begin{claim}
\label{r_k}
For any $a \in G(K)$, 
let us put 
\[
h_{k+1,\, a} = [a^{-1} g_{k,\, m} a,\ r_{k+1}]. 
\] 
Then $\{ r_1, \dots, r_k, r_{k+1} \} \subset \mathcal{S}_K(h_{k+1,\, a})$. 
\end{claim}

We remark that $r_{k+1}$ is used as a rational number or a group element representing the surgery slope
in the remainder of this section.

\begin{proof}
\[
p_{r_i}(h_{k+1,\, a}) 
= [p_{r_i}(a)^{-1}p_{r_i}(g_{k,\, m})p_{r_i}(a),\ p_{r_i}(r_{k+1})] = [1,\ p_{r_i}(r_{k+1})] = 1 
\]
for $i = 1, \dots, k$,
and 
\begin{align*}
 p_{r_{k+1}}(h_{k+1,\, a}) &=  [p_{r_{k+1}}(a)^{-1}p_{r_{k+1}}(g_{k,\, m})p_{r_{k+1}}(a),\ p_{r_{k+1}}(r_{k+1})]\\
&= [p_{r_{k+1}}(a)^{-1} p_{r_{k+1}}(g_{k,\, m})p_{r_{k+1}}(a),\ 1] = 1.
\end{align*}
Thus $\{ r_1, \dots, r_k, r_{k+1} \} \subset \mathcal{S}_K(h_{k+1,\, a})$. 
\end{proof}

\medskip

Let us take a slope $s \in \mathcal{S}$. 
Then since $\mathcal{R} \cap \mathcal{S} = \emptyset$, 
for $r_{k+1} \in \mathcal{R}_{k+1} \subset \mathcal{R}$ we have $s \ne r_{k+1}$. 
Since $s$ is not a finite surgery slope, 
by Proposition~\ref{slope} $p_s(r_{k+1}) \ne 1$. 
Besides, 
since $\mathcal{S}_K(g_{k,\, m}) \subset \mathbb{Q} - \mathcal{S}$, 
we have the following 
\[
p_s(g_{k,\, m}) \ne 1\quad \textrm{for}\quad s \in \mathcal{S} = \{ s_1, \dots, s_m\}.
\]
We will use this property of $g_{k, m}$ several times in the following discussion.

Note that 
\[
p_s(h_{k+1,\, a}) 
=1\quad \textrm{if and only if}\quad [p_s(a^{-1} g_{k,\, m} a),\ p_s(r_{k+1})] = 1. 
\] 

Claim~\ref{alpha_non-commute} below is crucial in the following procedure, 
and we will apply this several times.  
The proof of this claim is a little bit long, so we defer its proof until we finish the proof of Lemma~\ref{separation}. 
In fact, we present Lemma~\ref{conjugate_non-commute} (in slightly generalized form) after completing the proof of Lemma~\ref{separation}, from which Claim~\ref{alpha_non-commute} immediately follows.  

\begin{claim}
\label{alpha_non-commute}
There exists $\alpha \in \pi_1(K(s))$ such that 
\[
[\alpha^{-1} p_s(g_{k,\, m})\alpha,\ p_s(r_{k+1})] \ne 1.
\] 
\end{claim}

\medskip

Let us return to a proof of the second step. 

\medskip

${\bf (1)}$\ Following Claim~\ref{alpha_non-commute} 
there exists $\alpha_1 \in \pi_1(K(s_1))$ such that 
\[
[\alpha_1^{-1} p_{s_1}(g_{k,\, m}) \alpha_1,\ p_{s_1}(r_{k+1})] \ne 1. 
\] 
Since $p_{s_1} \colon G(K) \to \pi_1(K(s_1))$ is surjective, 
we have $a_1 \in G(K)$ such that  $p_{s_1}(a_1) = \alpha_1$. 
Then we have 
\[
p_{s_1}( [a_1^{-1} g_{k,\, m} a_1,\ r_{k+1}] ) \ne 1. 
\]

Let us write 
\[
h_{k+1,\, a_1} = [a_1^{-1} g_{k,\, m} a_1,\  r_{k+1}] \in G(K), 
\]
which satisfies 
$p_{s_1}(h_{k+1,\, a_1}) \ne 1$, i.e. 
$s_1 \not \in \mathcal{S}_K(h_{k+1,\, a_1})$, 
in particular, $h_{k+1,\, a_1}$ is non-trivial. 
By Claim~\ref{r_k} $p_{r}(h_{k+1,\, a_1}) = 1$ for $r \in \{ r_1, \dots, r_k, r_{k+1}\}$. 
Hence, 
we put 
\[
g_{k+1,\, 1} = h_{k+1,\, a_1} \in [G(K), G(K)] 
\]
so that
\[
\mathcal{R}_{k+1} = \{r_1,\ldots,r_{k}, r_{k+1} \} \subset \mathcal{S}_K(g_{k+1,\, 1}) \subset \mathbb{Q} - \{ s_1 \}.
\]

\medskip

In the following we will apply step-by-step modification of $g_{k+1,\, 1}$ to obtain $g_{k+1,\, m} \in [G(K), G(K)]$ which satisfies 
\[
\mathcal{R}_{k+1}  = \{r_1,\ldots,r_{k}, r_{k+1} \} \subset \mathcal{S}_K(g_{k+1,\, m}) \subset \mathbb{Q} - \{ s_1, \dots, s_m\}.
\]

\medskip

${\bf (2)}$\  We will construct $g_{k+1,\, 2} \in [G(K), G(K)]$ satisfying 
\[
\mathcal{R}_{k+1} = \{r_1,\ldots,r_{k}, r_{k+1} \} \subset \mathcal{S}_K(g_{k+1,\, 2}) \subset \mathbb{Q} - \{ s_1, s_2 \}.
\]

\medskip

${\bf (2.1)}$\   If $s_2 \not\in \mathcal{S}_K(g_{k+1,\, 1})$, i.e. $p_{s_2}(g_{k+1,\, 1}) \ne 1$ neither, 
then put simply
$g_{k+1,\, 2} = g_{k+1,\, 1} \in [G(K), G(K)]$. 
Then $g_{k+1,\, 2}$ satisfies the property in ${\bf(2)}$.  

\medskip

${\bf (2.2)}$\  If $s_2 \in \mathcal{S}_K(g_{k+1,\, 1})$, i.e. $p_{s_2}(g_{k+1,\, 1}) = 1$, 
then we need to modify $g_{k+1,\, 1}$ further to obtain $g_{k+1,\, 2}$ with the above property in ${\bf (2)}$. 

\medskip

Claim~\ref{alpha_non-commute},  
together with the surjectivity of $p_{s_2} \colon G(K) \to \pi_1(K(s_2))$,  
enables us to choose $a_2 \in G(K)$ so that $p_{s_2}(a_2) = \alpha_2 \in \pi_1(K(s_2))$ satisfies 
\[
[\alpha_2^{-1} p_{s_2}(g_{k,\, m}) \alpha_2,\  p_{s_2}(r_{k+1})] \ne 1,\ \textrm{equivalently}\ p_{s_2}([a_2^{-1} g_{k,\, m} a_2,\ r_{k+1}]) \ne 1. 
\]

Then put 
\[
h_{k+1,\, a_2} = [a_2^{-1} g_{k,\, m} a_2,\ r_{k+1}] \in G(K), 
\] 
which satisfies $p_{s_2}(h_{k+1,\, a_2}) \ne 1$, in particular, $h_{k+1,\, a_2}$ is non-trivial.  
Since $p_r(g_{k,\, m}) = 1$ for $r \in \mathcal{R}_k$ and obviously $p_{r_{k+1}}(r_{k+1}) = 1$, 
\[
p_r(h_{k+1,\, a_2}) 
= [p_r(a_2)^{-1}p_r(g_{k,\, m})p_r(a_2),\ p_r(r_{k+1})] 
= 1\  \textrm{for}\  r \in \mathcal{R}_{k+1}. 
\] 

Recall also that $p_{s_1}(g_{k+1,\, 1}) \ne 1$ and  $p_{s_2}(g_{k+1,\, 1}) = 1$. 

\begin{claim}
\label{q_1}
We may choose an integer $q_1$ so that 
\[
p_{s_1}\left((g_{k+1,\, 1})^{q_1} h_{k+1,\, a_2}\right) \ne 1 \in \pi_1(K(s_1)).
\] 
\end{claim}

\begin{proof}
If $p_{s_1}\left((g_{k+1,\, 1}) h_{k+1,\, a_2}\right) \ne 1$, then we choose $q_1 = 1$. 
Assume that 
\[
p_{s_1}((g_{k+1,\, 1}) h_{k+1,\, a_2}) 
= p_{s_1}(g_{k+1,\, 1}) p_{s_1}(h_{k+1,\, a_2}) = 1,\  
\textrm{i.e.}\  p_{s_1}(h_{k+1,\, a_2}) = p_{s_1}(g_{k+1,\, 1})^{-1}.
\] 

The argument is divided into two cases.
If $p_{s_1}(g_{k+1,\, 1})$ is not a torsion, then set $q_1$ to be any integer at least $3$.
Then 
$p_{s_1}((g_{k+1,\, 1})^{q_1} h_{k+1,\, a_2}) = p_{s_1}(g_{k+1,\, 1})^{q_1} p_{s_1}( h_{k+1,\, a_2}) = p_{s_1}(g_{k+1,\, 1})^{q_1-1} \ne 1$.

Otherwise, $p_{s_1}(g_{k+1,\, 1})$ is a torsion element of order $t_1\ge 2$.
Then 
 take $q_1$ so that $q_1-1$ is not a multiple of $t_1$. 
Then $p_{s_1}((g_{k+1,\, 1})^{q_1} h_{k+1,\, a_2}) \ne 1$.  
\end{proof}

\medskip

Furthermore, 
\[
p_{s_2}\left((g_{k+1,\, 1})^{q_1} h_{k+1,\, a_2}\right) = p_{s_2}(g_{k+1,\, 1})^{q_1} p_{s_2}(h_{k+1,\, a_2}) =  p_{s_2}(h_{k+1,\, a_2}) \ne 1,
\]
because $p_{s_2}(g_{k+1,\, 1}) = 1$. 
Hence, 
\[
\mathcal{S}_K((g_{k+1,\, 1})^{q_1} h_{k+1,\, a_2}) \subset \mathbb{Q} - \{ s_1, s_2 \}.
\]

By construction $p_r(g_{k+1,\, 1}) = p_r(h_{k+1,\, a_2}) = 1$, 
and hence,  
\[
p_r\left((g_{k+1,\, 1})^{q_1} h_{k+1,\, a_2}\right) 
= p_r(g_{k+1,\, 1})^{q_1} p_r(h_{k+1,\, a_2}) = 1.
\] 
for any slope $r \in \{ r_1, \dots, r_k, r_{k+1} \}$. 
This shows $\mathcal{R}_{k+1} \subset \mathcal{S}_K\left((g_{k+1,\, 1})^{q_1} h_{k+1,\, a_2}\right)$. 

Now put 
\[
g_{k+1,\, 2} = (g_{k+1,\, 1})^{q_1} h_{k+1,\, a_2} \in G(K)
\]
so that 
\[
\mathcal{R}_{k+1} = \{r_1,\ldots,r_{k}, r_{k+1} \} \subset \mathcal{S}_K(g_{k+1,\, 2}) \subset \mathbb{Q} - \{ s_1, s_2 \}.
\]

Note that since both $g_{k+1,\, 1}$ and $h_{k+1,\, a_2}$ belong to $[G(K), G(K)]$, 
so does $g_{k+1,\, 2}$. 

\medskip

Assume that we have an element $g_{k+1, i}$ satisfying 

\medskip

${\bf (i)}\ \mathcal{R}_{k+1} = \{r_1,\ldots,r_{k}, r_{k+1} \} \subset \mathcal{S}_K(g_{k+1,\, i}) \subset \mathbb{Q} - \{ s_1, \dots, s_i \}$\ 
$(2 \le i \le m-1)$. 

Now we will construct element $g_{k+1, i+1} \in G(K)$ satisfying 

\medskip

$\mathcal{R}_{k+1} = \{r_1,\ldots,r_{k}, r_{k+1} \} \subset \mathcal{S}_K(g_{k+1,\, i+1}) \subset \mathbb{Q} - \{ s_1, \dots, s_i, s_{i+1} \}$. 

\medskip

${\bf (i+1, 1)}$\   
If $s_{i+1} \not\in \mathcal{S}_K(g_{k+1, i})$, i.e. $p_{s_{i+1}}(g_{k+1, i}) \ne 1$ neither, 
then put simply 
$g_{k+1, i+1} = g_{k+1, i}$ so that 
$\mathcal{R}_{k+1} = \{r_1,\ldots,r_{k}, r_{k+1} \} \subset \mathcal{S}_K(g_{k+1,\, i+1}) \subset \mathbb{Q} - \{ s_1, \dots, s_i, s_{i+1} \}$. 

\medskip

${\bf (i+1, 2)}$\ 
If $s_{i+1} \in \mathcal{S}_K(g_{k+1, i})$, i.e. $p_{s_{i+1}}(g_{k+1, i}) = 1$, 
then we will modify $g_{k+1, i}$ further to obtain $g_{k+1, i+1}$ with the property in ${\bf (i)}$. 
By the assumption $s_j$ is not a torsion surgery slope for $j \ge 3$. 

\medskip

Again, following Claim~\ref{alpha_non-commute},  
together with the surjectivity of $p_{s_{i+1}} \colon G(K) \to \pi_1(K(s_{i+1}))$, 
we find $a_{i+1} \in G(K)$ so that $p_{s_{i+1}}(a_{i+1}) = \alpha_{i+1} \in \pi_1(K(s_{i+1}))$ satisfies 
\[
[\alpha_{i+1}^{-1} p_{s_{i+1}}(g_{k,\, m}) \alpha_{i+1},\ p_{s_{i+1}}(r_{k+1})] \ne 1,\ \textrm{i.e.}\ p_{s_{i+1}}( [a_{i+1}^{-1} g_{k,\, m} a_{i+1},\ r_{k+1}] ) \ne 1. 
\]
Then put 
\[
h_{k+1,\, a_{i+1}} = [a_{i+1}^{-1}g_{k,\, m} a_{i+1},\ r_{k+1}] \in G(K), 
\]
which is non-trivial and $p_{s_{i+1}}(h_{k+1,\, a_{i+1}}) \ne 1$. 

For convenience, we collect some properties of $g_{k+1,\, i}$ and $h_{k+1,\, a_{i+1}}$. 

\begin{itemize}
\item $p_r(g_{k+1,\, i}) = p_r(h_{k+1,\, a_{i+1}}) = 1$ for every slope $r \in \mathcal{R}_{k+1} = \{ r_1, \dots, r_k, r_{k+1} \}$. 
\item $p_{s}(g_{k+1,\, i}) \ne 1$ for $s \in \{ s_1, \dots, s_i \}$. 
\item $p_{s_{i+1}}(g_{k+1,\, i}) = 1$ and $p_{s_{i+1}}(h_{k+1,\, a_{i+1}}) \ne 1$. 
\end{itemize}

\medskip

For slopes $s_j \in \mathcal{S}$, 
we set an integer $n_j$ as follows. 

If $\pi_1(K(s_j))$ is torsion free, then put $n_j = 1$. 
Suppose that $s_j$ is a torsion surgery slope. 
Then $s_j$ is either a finite surgery slope or a reducing surgery slope. 
For such a surgery slope $s_j$, we set $n_j$ as follows. 
If $s_j$ is a finite surgery slope, then let $n_j = |\pi_1(K(s_j))|$.   
If $s_j$ is a reducing surgery slope, 
then $K(s_j)$ has at most three prime factors \cite{How}  $M_{j, 1},\  M_{j, 2},\ M_{j, 3}$  
($M_{j, 3}$ may be $S^3$).  
Let $m_{j, 1} = |\pi_1(M_{j, 1})|$ if $|\pi_1(M_{j, 1})| < \infty$, otherwise $m_{j, 1} = 1$. 
Set $m_{j, 2}$ and $m_{j, 3}$ in a similar fashion. 
Then put $n_j = m_{j, 1}m_{j, 2}m_{j, 3}$. 

Recall that, by the assumption, 
at most two slopes $s_1$ and $s_2$ can be torsion surgery slopes, 
hence $n_j = 1$ for $j \ge 3$.

\begin{claim}
\label{q_2}
There exists an integer $q_i$ such that 
\[
p_{s_1}\left((g_{k+1,\, i})^{q_i n_1n_2 + 1} (h_{k+1,\, a_{i+1}})^{n_1 n_2}\right) \ne 1 \in \pi_1(K(s_1)),\  \textrm{and}
\]
\[
p_{s_2}\left((g_{k+1,\, i})^{q_i n_1n_2+ 1} (h_{k+1,\, a_{i+1}})^{n_1n_2}\right) \ne 1 \in \pi_1(K(s_2)).
\]
\end{claim}

\begin{proof}
We need to consider three cases: 
\begin{enumerate}
\renewcommand{\labelenumi}{(\roman{enumi})}
\item 
$s_1$ and $s_2$ are not torsion surgery slopes, 
\item
$s_1$ is a torsion surgery slope, but $s_2$ is not a torsion surgery slope, or 
\item 
$s_1$ and $s_2$ are torsion surgery slopes. 
\end{enumerate}

\medskip

(i) 
Assume that  
\[
p_{s_1}\left((g_{k+1,\, i})^{\ell n_1n_2 + 1} (h_{k+1,\, a_{i+1}})^{n_1n_2}\right) 
= p_{s_1}(g_{k+1,\, i})^{\ell n_1n_2 + 1} p_{s_1}(h_{k+1,\, a_{i+1}})^{n_1n_2} = 1,\  \textrm{and}
\]
\[
p_{s_1}\left((g_{k+1,\, i})^{\ell' n_1n_2 + 1} (h_{k+1,\, a_{i+1}})^{n_1n_2}\right) 
= p_{s_1}(g_{k+1,\, i})^{\ell' n_1n_2 + 1} p_{s_1}(h_{k+1, a_{i+1}})^{n_1n_2} = 1
\]
for some integers $\ell, \ell'$. 
Then we have 
$p_{s_1}(g_{k+1,\, i})^{(\ell - \ell') n_1n_2} = 1$ independently of whether $p_{s_1}(h_{k+1, a_{i+1}}) = 1$ or not. 
Since $p_{s_1}(g_{k+1,\, i}) \ne 1$ and $\pi_1(K(s_1))$ is torsion free, 
$\ell = \ell'$. 
This means that $p_{s_1}\left((g_{k+1,\, i})^{x n_1n_2 + 1} (h_{k+1,\, a_{i+1}})^{n_1n_2}\right) = 1$ for at most one integer $x = \ell_1 \ge 1$. 
Similarly $p_{s_2}\left((g_{k+1,\, i})^{y n_1n_2 + 1} (h_{k+1,\, a_{i+1}})^{n_1n_2}\right) = 1$ for at most one integer $y = \ell_2 \ge 1$. 

We choose $q_i \ne \ell_1, \ell_2$ so that $p_{s_j}\left((g_{k+1,\, i})^{q_i n_1n_2 + 1} (h_{k+1,\, a_{i+1}})^{n_1n_2}\right) \ne 1 \in \pi_1(K(s_j))$ for $j = 1, 2$. 

\medskip

(ii) By the definition of $n_1$, 
\[
p_{s_1}(g_{k+1,\, i})^{n_1} = 1\quad \textrm{and}\quad p_{s_1}( h_{k+1,\, a_{i+1}})^{n_1} = 1.
\] 
Thus we have 
\begin{align*}
p_{s_1}\left((g_{k+1,\, i})^{x n_1n_2 + 1} (h_{k+1,\, a_{i+1}})^{n_1n_2}\right)
& = p_{s_1}(g_{k+1,\, i})^{x n_1n_2 + 1} p_{s_1}(h_{k+1,\, a_{i+1}})^{n_1n_2} \\
& = p_{s_1}(g_{k+1,\, i})\ne 1
\end{align*}
for any integer $x \ge 1$. 

For non-torsion surgery slope $s_2$, as observed in (i), 
\[
p_{s_2}\left((g_{k+1,\, i})^{y n_1n_2+ 1} (h_{k+1,\, a_{i+1}})^{n_1n_2}\right) = 1
\] for at most one integer $y = \ell_2$. 

So we may choose $x = q_i \ne \ell_2$ so that 
\[
p_{s_1}\left((g_{k+1,\, i})^{q_i n_1n_2 + 1} (h_{k+1,\, a_{i+1}})^{n_1n_2}\right) \ne 1, \quad \textrm{and} 
\]
\[
p_{s_2}\left((g_{k+1,\, i})^{q_i n_1n_2+ 1} (h_{k+1,\, a_{i+1}})^{n_1n_2}\right) \ne 1.
\] 

\medskip

(iii) By the definition of $n_1$ and $n_2$,  
\[
p_{s_1}(g_{k+1,\, i})^{n_1} = 1,\  p_{s_1}( h_{k+1,\, a_{i+1}})^{n_1} = 1\ \textrm{and}
\]
\[
p_{s_2}(g_{k+1,\, i})^{n_2} = 1,\  p_{s_2}( h_{k+1,\, a_{i+1}})^{n_2} = 1. 
\]
Hence, we have
\begin{align*}
p_{s_1}\left((g_{k+1,\, i})^{x n_1n_2 + 1} (h_{k+1,\, a_{i+1}})^{n_1n_2}\right) 
& = p_{s_1}(g_{k+1,\, i})^{x n_1n_2 + 1} p_{s_1}(h_{k+1,\, a_{i+1}})^{n_1n_2} \\
& = p_{s_1}(g_{k+1,\, i})
\ne 1 \in \pi_1(K(s_1)),
\end{align*}
and 
\begin{align*}
p_{s_2}\left((g_{k+1,\, i})^{x n_1n_2+ 1} (h_{k+1, a_{i+1}})^{n_1n_2}\right) 
& = p_{s_2}(g_{k+1,\, i})^{x n_1n_2+ 1} p_{s_2}(h_{k+1, a_{i+1}})^{n_1n_2} \\
& = p_{s_2}(g_{k+1,\, i})
\ne 1 \in \pi_1(K(s_2))
\end{align*}
for any integer $x \ge 1$.  
So we may choose a desired integer $q_i \ge 1$. 
\end{proof}

\medskip

Furthermore, since $p_{s_{i+1}}(g_{k+1,\, i}) = 1$,\ $p_{s_{i+1}}(h_{k+1,\, a_{i+1}}) \ne 1$ and $\pi_1(K(s_{i+1}))$ is torsion free, 
\begin{align*}
p_{s_{i+1}}\left((g_{k+1,\, i})^{q_i n_1n_2 + 1} (h_{k+1,\, a_{i+1}})^{n_1n_2}\right) 
& = p_{s_{i+1}}(g_{k+1,\, i})^{q_i n_1n_2 + 1} p_{s_{i+1}}(h_{k+1,\, a_{i+1}})^{n_1n_2} \\
& = p_{s_{i+1}}(h_{k+1,\, a_{i+1}})^{n_1n_2}
\ne 1.
\end{align*}

On the other hand, 
since $p_r(g_{k+1,\, i}) = p_r(h_{k+1,\, a_{i+1}}) = 1$ for $r \in \{ r_1, \dots, r_k, r_{k+1} \} = \mathcal{R}_{k+1}$, 
\begin{align*}
p_r\left((g_{k+1,\, i})^{q_{i} n_1n_2+1} (h_{k+1,\, a_{i+1}})^{n_1n_2}\right) 
& = p_r(g_{k+1,\, i})^{q_{i} n_1n_2+1} p_r(h_{k+1,\, a_{i+1}})^{n_1n_2} \\
& = 1 
\end{align*}
for $r \in \mathcal{R}_{k+1}$. 
This shows $\mathcal{R}_{k+1} \subset \mathcal{S}_K\left((g_{k+1,\, i})^{q_{i} n_1n_2+1} (h_{k+1,\, a_{i+1}})^{n_1n_2}\right)$. 

Hence, we obtain a non-trivial element 
\[
g_{k+1,\, {i+1}} = (g_{k+1,\, i})^{q_{i}n_1n_2+1} (h_{k+1,\, a_{i+1}})^{n_1n_2} \in [G(K), G(K)] 
\]
which satisfies 
\[
\mathcal{R}_{k+1} = \{r_1,\ldots,r_{k}, r_{k+1} \} \subset \mathcal{S}_K(g_{k+1,\, {i+1}}) \subset \mathbb{Q} - \{ s_1, \dots s_{i+1} \}.
\]
Therefore we have an element $g_{k+1, m} \in G(K)$ which satisfies
\[
\mathcal{R}_{k+1} = \{r_1,\ldots,r_{k}, r_{k+1} \} \subset \mathcal{S}_K(g_{k+1,\, m}) \subset \mathbb{Q} - \{ s_1, \dots s_{m} \} = \mathcal{S}.
\]

This completes \textbf{Step 2}. 

\medskip

Continue the above procedure to obtain a non-trivial element $g_{n,\, m} \in [G(K), G(K)]$ with 
\[
\mathcal{S}_n = \{r_1,\ldots,r_{n-1}, r_n \} \subset \mathcal{S}_K(g_{n,\, m}) \subset \mathbb{Q} - \{ s_1, \dots, s_{m} \}.
\]
Then $g = g_{n,\, m} \in [G(K), G(K)]$ is a desired element, 
and this completes a proof of Lemma~\ref{separation}.
\end{proof}

Now let us turn to prove Lemma~\ref{conjugate_non-commute} below, 
which immediately implies Claim~\ref{alpha_non-commute} used in the proof of Lemma~\ref{separation}. 

\begin{lemma}
\label{conjugate_non-commute}
Assume that $K(s)$ is not a Seifert fiber space. 
Then for any element $g \in G(K)$ and any slope element $r \in G(K)$ such that 
$p_s(g)\ne 1$ and $p_s(r) \ne 1$, 
there exists $\alpha \in \pi_1(K(s))$ such that 
\[
[\alpha^{-1} p_s(g) \alpha,\ p_s(r)] \ne 1. 
\] 
\end{lemma}

\begin{proof}
Suppose for a contradiction that 
$[\alpha^{-1} p_s(g) \alpha,\ p_s(r)] = 1$ for all $\alpha \in \pi_1(K(s))$. 
This then implies 
\[
[\alpha^{-1} p_s(g)^{-1} \alpha,\  p_s(r)] = [(\alpha^{-1} p_s(g) \alpha)^{-1},\ p_s(r)] =1 
\]
for all $\alpha \in \pi_1(K(s))$ as well. 
Hence, the normal closure $\langle \! \langle p_s(g) \rangle\!\rangle$ is contained in $Z(p_s(r))$;
$\langle \! \langle p_s(g) \rangle\!\rangle \subset Z(p_s(r))$.

First we consider the case where $K(s)$ is reducible. 
Then $K(s)$ is ether $S^2 \times S^1$ or non-prime. 
Since we assume that $s$ is not a Seifert surgery slope, the former does not happen.
Thus $K(s)$ is non-prime, hence its fundamental group $\pi_1(K(s))$ is decomposed as a free product $A \ast B$. (Note that $A$ or $B$ may be decomposable.)
By the structure of the commuting elements of the free product (\cite[Corollary~4.1.6 and p.196, 28]{MKS}), 
we have an element $x \in A*B$ such that $x^{-1} Z(p_s(r)) x\subset A$ or $x^{-1} Z(p_s(r)) x\subset B$, 
or $Z(p_s(r))$ is cyclic. 
 
In the first case, 
$\langle \! \langle p_s(g) \rangle\!\rangle = x^{-1} \langle \! \langle p_s(g) \rangle\!\rangle x \subset x^{-1} Z(p_s(r)) x \subset A$. 
Take $a \ (\ne 1)  \in \langle \! \langle p_s(g) \rangle\!\rangle \subset A$. 
Then $b^{-1} ab \not\in A$ for $b\ (\ne 1)\in B$ \cite[Theorem~4.2]{MKS}, 
in particular not in $\langle \! \langle p_s(g) \rangle\!\rangle$, a contradiction. 
Similarly the second case does not occur. 

So $Z(p_s(r))$ is cyclic, and thus $\langle \! \langle p_s(g) \rangle\!\rangle \subset Z(p_s(r))$ is also cyclic. 
If $A \cong B \cong \mathbb{Z}_2$, 
then $\pi_1(K(s)) \cong \mathbb{Z}_2 * \mathbb{Z}_2$. 
Abelianizing this we have $H_1(K(s)) \cong \mathbb{Z}_2 \oplus \mathbb{Z}_2$, contradicting $H_1(K(s))$ being cyclic. 
Hence we may assume $A \not\cong Z_2$.  
Let us write, if necessary taking a conjugation of $A*B$, 
$p_s(g) = a_1b_1\cdots a_kb_k$ in which $a_i \in A, \ b_i \in B$ are non-trivial and $k \ge 1$ \cite[Theorem~4.2]{MKS}. 
Take $a \ (\ne 1) \in A$ so that $a \ne  a_1$. 
(This is possible because $A \ne \mathbb{Z}_2$. 
In fact, for the Seifert fiber space $\mathbb{R}P^3 \# \mathbb{R}P^3$,  
$\pi_1(\mathbb{R}P^3 \# \mathbb{R}P^3) \cong \mathbb{Z}_2 * \mathbb{Z}_2$ has an infinite cyclic normal group generated by a regular fiber.)
Then 
\[
(a_1 a^{-1})^{-1} p_s(g)  (a_1 a^{-1})
= (a a_1^{-1}) (a_1 b_1 \cdots a_kb_k)  (a_1 a^{-1})
= ab_1a_2b_2\cdots a_kb_k  (a_1 a^{-1}),
\] 
and apply \cite[Theorem~4.2]{MKS} to see that any power of this element  
\[
(a b_1 a_2 \cdots b_k a_1)(b_1 \cdots a_1) \cdots (b_1\cdots a_1)(b_1\cdots a_1a^{-1})
\] 
cannot be $(a_1b_1\cdots b_k)^n$ for any integer $n$. 
This shows that $\langle \! \langle p_s(g) \rangle\!\rangle$ is not normal, a contradiction. 

Thus we may assume that $K(s)$ is irreducible.

\medskip

It follows from the structure theorem of centralizers \cite[Theorems~2.5.1 and 2.5.2]{AFW} that 
we have the following: 
\begin{enumerate}
\renewcommand{\labelenumi}{(\roman{enumi})}
\item 
$Z(p_s(r)) \cong \mathbb{Z}, \mathbb{Z} \oplus \mathbb{Z}$, or, \mbox{the Klein bottle group}. 
\item There is a Seifert piece $M$ of the torus decomposition \cite{JS,Jo,Hat} of $K(s)$ such that, 
up to conjugation of $\pi_1(K(s))$, 
$p_s(r)$ is a power of an element represented by a regular fiber in $M$ and  
$Z(p_s(r))$ is the canonical subgroup of $\pi_1(M)$, in particular $Z(p_s(r)) \subset \pi_1(M)$.  
\end{enumerate} 

\medskip 

\noindent
{\bf Case (i)}.\\
Since $\langle \! \langle p_s(g) \rangle\!\rangle$ is a subgroup of $Z(p_s(r))$, 
$\langle \! \langle p_s(g) \rangle\!\rangle $ is isomorphic to either $\mathbb{Z},\mathbb{Z} \oplus \mathbb{Z}$, or the Klein bottle group.
If $\langle \! \langle p_s(g) \rangle\!\rangle$ has finite index in $\pi_1(K(s))$, then $\pi_1(K(s))$ is virtually abelian. 
By the assumption that $K(s)$ is not Seifert fibered, it has infinite fundamental group and the abelian subgroup (of finite index) is also infinite. 
So the finite cover of $K(s)$ associated to the abelian subgroup is $S^1 \times S^1 \times S^1$ \cite[p.25, Table~2]{AFW}.  
Recall that $K(s)$ is irreducible and $\pi_1(K(s))$ is infinite.
It follows from the last paragraph of p.35 in \cite{Scott} that $K(s)$ is also Seifert fiber space, contradicting the assumption.

Therefore $\langle \! \langle p_s(g) \rangle\!\rangle \ (\subset Z(p_s(r)))$ is a finitely generated subgroup of infinite index, so we may apply \cite{HJ}, \cite[p.118 (L9)]{AFW} to see that 

\begin{enumerate}
\renewcommand{\labelenumi}{(\alph{enumi})}
\item
$K(s)$ is a Seifert fiber space and  $\langle \! \langle p_s(g) \rangle\!\rangle$ is a subgroup generated by a regular fiber in $K(s)$, 
\item
$K(s)$ fibers over $S^1$ with surface fiber $\Sigma$ and $\langle \! \langle p_s(g) \rangle\!\rangle$ is a finite index subgroup of $\pi_1(\Sigma)$, or 
\item
$K(s)$ is the union of two twisted $I$--bundles over a compact connected  (necessarily non-orientable) surface $\Sigma$, 
and $\langle \! \langle p_s(g) \rangle\!\rangle$ is a finite index subgroup of $\pi_1(\Sigma)$. 
\end{enumerate}

The assumption that $s$ is not a Seifert surgery slope excludes the first possibility (a).

In the second case (b), for homological reason, $s=0$. 
Let us write the slope element $r \in G(K)$ as $\mu^m \lambda^n$ for some relatively prime integers $m$ and $n$. 
Since $p_0(\lambda)  =1$ in $\pi_1(K(0))$,  
$p_0(r) = p_0(\mu)^m \in \pi_1(K(0))$.  
Moreover, since $\langle \! \langle p_0(g) \rangle\!\rangle$ is a finite index subgroup of $\pi_1(\Sigma)$ of a closed orientable surface $\Sigma$ 
and $\langle \! \langle p_0(g) \rangle\!\rangle \cong \mathbb{Z}$, $\mathbb{Z} \oplus \mathbb{Z}$, or the Klein bottle group, 
$\pi_1(\Sigma)\cong \mathbb{Z} \oplus \mathbb{Z}$ and $\langle \! \langle p_0(g) \rangle\!\rangle \cong \mathbb{Z} \oplus \mathbb{Z}$. 
Since $\langle \! \langle p_0(g) \rangle\!\rangle \subset Z(p_0(r))$ and $Z(p_0(r))\cong \mathbb{Z}$, $\mathbb{Z} \oplus \mathbb{Z}$, 
or the Klein bottle group, 
this shows that $Z(p_0(r)) \cong \mathbb{Z} \oplus \mathbb{Z}$ or the Klein bottle group and that $\langle \! \langle p_0(g) \rangle\!\rangle $ is a finite index subgroup of $Z(p_0(r))$. 
Since $p_0(r) \in Z(p_0(r))$, $p_0(r)^{a} \in \langle \! \langle p_0(g) \rangle\!\rangle $ for some $a > 0$. 
Then $p_0(r)^a = p_0(\mu)^{ma} \in \langle \! \langle p_0(g) \rangle\!\rangle \subset \pi_1(\Sigma) = [\pi_1(K(0)),\pi_1(K(0))]$. 
This then implies that 
$[p_0(r)^a] = [p_0(\mu)^{ma}] = ma [p_0(\mu)] = 0$ in $H_1(K(0)) =  \langle [p_0(\mu)] \rangle = \mathbb{Z}$. 
Hence $m = 0$ and the slope element $r$ is longitudinal. 
Then $p_s(r) = p_0(r) = 1$, contradicting the assumption. 

\medskip

In the third case (c), then $K(s)=X\cup Y$ is the union of two twisted $I$--bundles $N(\Sigma)$ over the non-orientable surface $\Sigma$.
Write $\widetilde{\Sigma} = \partial N_{\Sigma}$, 
which is the $\partial I$--subbundle of $N_{\Sigma}$. 
Then $\pi_1(K(s)) = \pi_1(N_{\Sigma}) \ast_{\pi_1(\widetilde{\Sigma})} \pi_1(N_{\Sigma})$. 
Since $\pi_1(N_{\Sigma})/\pi_1(\widetilde{\Sigma}) \cong \mathbb{Z}_2$, we have 
\[
1 \to \pi_1(\widetilde{\Sigma}) \to \pi_1(K(s)) \to \mathbb{Z}_2 \ast \mathbb{Z}_2 \to 1.
\]
This means that a cyclic group $H_1(K(s))$ has an epimorphism to $\mathbb{Z}_2 \oplus \mathbb{Z}_2$, a contradiction. 

\medskip

\noindent
{\bf Case (ii)}.\\
Take an element $h \in \pi_1(K(s))$ so that $h^{-1} p_s(r) h \in \pi_1(M)$ and 
$Z(h^{-1} p_s(r) h)$ is the canonical subgroup of  $\pi_1(M)$. 
Then we have 
\[
\langle \! \langle p_s(g) \rangle\!\rangle = h^{-1} \langle \! \langle p_s(g) \rangle\!\rangle h \subset 
h^{-1} Z(p_s(r)) h = Z(h^{-1} p_s(r) h) \subset \pi_1(M).
\] 

By the assumption $K(s)$ is not a Seifert fiber space, and hence the Seifert piece $M$ is not $K(s)$.  
Let $T$ be a component of $\partial M$ and $X$ a decomposing piece of $K(s)$ other than $M$ with $\partial X \supset T$. 

If $\langle \! \langle p_s(g) \rangle\!\rangle \subset \pi_1(T) \cong \mathbb{Z} \oplus \mathbb{Z}$, 
then $\langle \! \langle p_s(g) \rangle\!\rangle \cong \mathbb{Z}$ or $\mathbb{Z} \oplus \mathbb{Z}$ so we have done by Case (i).
So we may assume that $\langle \! \langle p_s(g) \rangle\!\rangle \not\subset \pi_1(T)$. 
Then we have an element $\gamma \in \langle \! \langle p_s(g) \rangle\!\rangle -\pi_1(T) \subset  \pi_1(M) - \pi_1(T)$.  

Now let us find $\alpha  \in \pi_1(X) - \pi_1(T)$. 
Assume for a contradiction that $\pi_1(X) = \pi_1(T)$. 
Then $\pi_1(X) \cong \mathbb{Z} \oplus \mathbb{Z}$ and  \cite[Theorem~10.6]{Hem} shows that $X$ is an $I$--bundle over a closed surface, 
precisely $X$ is either $S^1 \times S^1 \times I$ or the twisted $I$--bundle over the Klein bottle. 
The first case does not happen, 
and in the second case $\pi_1(T)$ is an index two subgroup of $\pi_1(X)$, and hence $\pi_1(X) \ne \pi_1(T)$, a contradiction. 
So there exists $\alpha  \in \pi_1(X) - \pi_1(T)$.  

Consider the element 
$\alpha^{-1} \gamma \alpha$. 
Since $\alpha \in \pi_1(X) - \pi_1(T)$ and $\gamma \in \pi_1(M) - \pi_1(T)$, 
we may apply \cite[Corollary~4.4.1]{MKS} to see that $\alpha^{-1} \gamma \alpha \not\in \pi_1(M)$ in 
$\pi_1(M \cup_T X) = \pi_1(M) \ast_{\pi_1(T)} \pi_1(X)$. 
Since $\pi_1(M \cup_T X)$ injects into $\pi_1(K(s))$, 
$\alpha^{-1} \gamma \alpha \not\in \pi_1(M)$ in $\pi_1(K(s))$ neither. 
This implies $\alpha^{-1} \gamma \alpha \not\in \langle \! \langle p_s(g) \rangle\!\rangle \subset \pi_1(M) \subset \pi_1(K(s))$ 
for $\gamma \in \langle \! \langle p_s(g) \rangle\!\rangle$. 
This contradicts $\langle \! \langle p_s(g) \rangle\!\rangle$ being normal in $\pi_1(K(s))$. 
\end{proof}

\bigskip

\section{Shrinking Lemma}
\label{shrink}

Recall that 
in general we have the following inequality (Propositions~\ref{cup_cap} and \ref{S_K(g)_S_K(g^n)}).
\[
\mathcal{S}_K(g) \cap \mathcal{S}_K(h) \subset \mathcal{S}_K(g^m) \cap \mathcal{S}_K(h^n)
\subset \mathcal{S}_K(g^m h^n)
\]
for any $m, n \ne 0$. 

The purpose in this Section is prove that the inequality may be replaced by the equality under some conditions. 

\begin{lemma_S_K_intersection}[Shrinking Lemma]
Let $K$ be a hyperbolic knot. 
Let $g$ be a non-peripheral element and $h$ a non-trivial element in $[G(K),G(K)]$. 
Assume that $\mathbb{Q} - \mathcal{S}_K(g)$ contains neither a finite surgery slope nor two reducing surgery slopes. 
Then there are infinitely many integers $n$ and a constant $N_n > 0$ such that the equality 
\[
\mathcal{S}_K(g) \cap \mathcal{S}_K(h) = \mathcal{S}_K(g^mh^n) 
\] 
holds for infinitely many integers $m \ge N_n$ for each $n$.
\end{lemma_S_K_intersection}

\begin{proof}
As we mentioned above, 
we have 
\[
\mathcal{S}_K(g) \cap \mathcal{S}_K(h) \subset \mathcal{S}_K(g^mh^n)
\]
for all integers $m$ and $n$ without any condition. 

Recall that if $\pi_1(K(s))$ has a torsion element, 
then $K(s)$ is either a spherical Seifert fiber space or a non-prime $3$--manifold. 
Furthermore, there are at most only finitely many such slopes.  
Let $s_i$ ($i = 1, \dots, k$) be a torsion surgery slope. 
If $p_{s_i}(h)$ is a non-trivial torsion element, 
then let $a_i$ be its order ($i = 1, \dots, k$). 
Let $n$ be a prime number coprime to $a_1, \dots, a_k$. 
Then  $p_{s}(h)^n \ne 1$ for any slope $s \in \mathbb{Q}$ whenever $s \not\in \mathcal{S}_K(h)$. 
In what follows we take and fix such a prime integer $n_0 > 0$. 

Let us show that there exists a constant $N_{n_0}$ such that 
\[
\mathcal{S}_K(g) \cap \mathcal{S}_K(h) \supset \mathcal{S}_K(g^mh^{n_0})
\]
for infinitely many integers $m \ge N_{n_0}$.

Indeed we show that if $s \not\in \mathcal{S}_K(g)$ or $s \not\in \mathcal{S}_K(h)$, 
then $s \not\in \mathcal{S}_K(g^mh^{n_0})$. 
We divide the argument into two cases depending upon $s \in \mathcal{S}_K(g) - \mathcal{S}_K(h)$ or $s \not\in \mathcal{S}_K(g)$. 

\medskip

\begin{claim}
\label{s_in_S}
For all integres $m$, we have $p_s(g^m h^{n_0}) \ne 1$ 
for any slope $s \in \mathcal{S}_K(g) - \mathcal{S}_K(h)$. 
\end{claim}

\begin{proof}
Since $s \in \mathcal{S}_K(g)$, 
we have $p_s(g^{m}h^{n_0}) = p_s(g)^m p_s(h)^{n_0} = p_s(h)^{n_0}$.  
On the other hand, since $s \not\in \mathcal{S}_K(h)$, 
by the choice of $n_0$, 
we have $p_s(h)^{n_0} \ne 1$, and hence $s \not\in \mathcal{S}_K(g^mh^{n_0})$. 
\end{proof}

\medskip
So we assume that $s \not\in \mathcal{S}_K(g)$. 

\begin{claim}
\label{s_not_in_S_non-hyperbolic}
There are infinitely many integers $m > 0$ such that $p_s(g^mh^{n_0}) \ne 1$ for any non-hyperbolic surgery slope $s \not \in \mathcal{S}_K(g)$. 
\end{claim}

\begin{proof}
Recall first that $\mathbb{Q} - \mathcal{S}_K(g)$ contains neither a finite surgery slope nor two reducing surgery slope by the assumption. 
Since a torsion surgery slope is either a finite surgery slope or a reducing surgery slope,  
there is at most one torsion surgery slope $t_0 \not\in \mathcal{S}_K(g)$. 

If $p_s(g)$ is not a torsion element, 
then there is at most one integer $m_s$ such that $p_s(g)^{m_s} = p_s(h^{-{n_0}})$, 
i.e. $p_s(g^{m_s}h^{n_0}) = 1$. 
Assume that $p_s(g)$ is a torsion element. 
Now we describe integers $m$ for which $p_s(g)^m = p_s(h^{-n_0})$. 
If we have no such an integer $m$, then for all integers $m$ we have $p_s(g^mh^{n_0}) \ne 1$. 
Otherwise, take a smallest such an integer $m_0$, i.e. 
$p_s(g)^{m_0} = p_s(h^{-n_0})$. 
If an integer $m\ne m_0$ also satisfies $p_s(g)^{m} = p_s(h^{-n_0})$, 
then $p_s(g)^{m}  = p_{s}(g)^{m_0}$, 
i.e. $p_{s}(g)^{m - m_0} = 1$. 
Since $p_{s}(g) \ne 1$, this means that $p_{s}(g)$ is a torsion element. 
Hence $s$ is the unique torsion slope $t_0 \in \mathbb{Q} - \mathcal{S}_K(g)$.
Let $b_0 \ge 2$ be the order of $p_{t_0}(g)$.  
Then $m - m_0$ is a multiple of $b_0$, 
i.e. $m = m_0 + k b_0$ for some integer $k$. 
Put $T_{n_0} = \{ m_0 + k b_0\}_{k \in \mathbb{Z}}$ so that 
$p_s(g)^{m} \ne p_s(h^{-n_0})$, i.e. $p_s(g^m h^{n_0}) \ne 1$ 
whenever $m \not\in T_{n_0}$.  

For each non-hyperbolic surgery slope $s$ except at most one torsion surgery slope $t_0$, 
there is at most one possible integer $m_s$ satisfying $p_s(g^{m_s} h^{n_0}) = 1$, 
and for the torsion surgery slope $t_0$ we have infinitely many positive integers $m \not\in T_{n_0}$ satisfying 
$p_s(g^{m_s} h^{n_0}) \ne 1$.  
Since there are only finitely many non-hyperbolic surgery slopes, 
there are infinitely many positive integers $m$ for which $p_s(g^{m} h^{n_0}) \ne 1$ as required. 
\end{proof}

\medskip

Following Claim~\ref{s_not_in_S_non-hyperbolic} let us denote by $M_{n_0}$ the infinite set of positive integers $m$ each of which satisfies 
$p_s(g^{m} h^{-n_0}) \ne 1$ for any non-hyperbolic surgery slope $s \not\in \mathcal{S}_K(g)$. 

\begin{claim}
\label{s_not_in_S_hyperbolic}
There exists a constant $C_{n_0}$ such that $p_s(g^mh^{n_0}) \ne 1$ for any hyperbolic surgery slope $s \not\in \mathcal{S}_K(g)$ whenever $m > C_{n_0}$. 
\end{claim}

\begin{proof}
Since $g$ is non-peripheral, 
it follows from Theorem~\ref{scl_bound}
we may take a constant $\delta_g > 0$ (depending only on $g$) so that for all hyperbolic surgery slopes $s$, 
\[
\mathrm{scl}_{\pi_1(K(s))}(p_s(g)) > \delta_g > 0\ \textrm{whenever}\ p_s(g) \ne 1,\ \textrm{i.e.}\  s \not\in \mathcal{S}_K(g).
\] 
Furthermore, 
since $h \in [G(K), G(K)]$, 
its stable commutator length satisfies $\mathrm{scl}_{G(K)}(h) < \infty$.  
Thus we may choose $m$ so that 
\[
m\, \delta_g > n_0\, \mathrm{scl}_{G(K)}(h)+\frac{1}{2}. 
\]
Put the constant $C_{n_0}$ as 
\[
C_{n_0} = \frac{n_0\, \mathrm{scl}_{G(K)}(h)+\frac{1}{2}}{\delta_g} > 0. 
\]

Let us separate the argument into two cases depending upon $g \in [G(K), G(K)]$ or not. 

\smallskip

\noindent
\textbf{Case 1.}\ 
$g \in [G(K), G(K)]$. \ 
Since $h$ also belongs to $[G(K), G(K)]$, 
$p_s(g), p_s(h) \in [\pi_1(K(s)), \pi_1(K(s))]$ and 
Lemmas~\ref{scl_product} and \ref{monotonicity} and the choice of $m > C_{n_0}$ imply  
\begin{align*}
\mathrm{scl}_{\pi_1(K(s))}(p_s(g^{m}h^{n_0})) & \ge m\, \mathrm{scl}_{\pi_1(K(s))}(p_s(g)) - n_0\,\mathrm{scl}_{\pi_1(K(s))}(p_s(h))-\frac{1}{2}\\
& > m\, \delta_{g} - n_0\, \mathrm{scl}_{G(K)}(h) - \frac{1}{2} > 0.
\end{align*}
This shows that $p_s(g^{m}h^{n_0}) \neq 1$, 
hence $s \not\in \mathcal{S}_K(g^mh^{n_0})$. 

\medskip

\noindent
\textbf{Case 2.}\  
$g \not\in [G(K), G(K)]$. \  
If $s \ne 0$, then $K(s)$ is a rational homology $3$--sphere and $\pi_1(K(s))/[\pi_1(K(s)),\ \pi_1(K(s))]$ is finite, 
so we may still apply Lemma~\ref{scl_product}. 
Since $m > C_{n_0}$, 
as in Case 1, we see that $p_s(g^mh^{n_0}) \ne 1$. 

It remains to consider the case where $s=0$ and $p_0(g) \not\in [\pi_1(K(0)),\ \pi_1(K(0))]$.  

We observe the following. 

\begin{claim}
\label{s=0}
If $p_0(g) \not\in [\pi_1(K(0)), \pi_1(K(0))]$,  
then $p_0(g^mh^{n_0}) \ne 1$ for all $m > 0$. 
\end{claim}

\begin{proof}
Suppose that $p_0(g^mh^{n_0}) = 1$, i.e. $p_0(g)^m = p_0(h)^{-n_0}$. 
Abelianizing this we have $m [p_0(g)] = -n_0 [p_0(h)] \in H_1(K(0)) = \mathbb{Z}$. 
Since $h \in [G(K), G(K)]$, we have $[p_0(h)] = 0 \in H_1(K(0))$. 
On the other hand, 
since $p_0(g) \not\in [\pi_1(K(0)), \pi_1(K(0))]$,  
$[p_0(g)] \ne 0$. 
Hence, $m = 0$.  
\end{proof}

Hence $p_s(g^mh^{n_0}) \ne 1$ whenever $m > C_{n_0}$. 
This completes a proof of Claim~\ref{s_not_in_S_hyperbolic}.
\end{proof}

Following Claims~\ref{s_in_S}, \ref{s_not_in_S_non-hyperbolic} and \ref{s_not_in_S_hyperbolic}, 
we may observe that there are infinitely many integers $m$ with $C_{n_0} < m \in M_{n_0}$ for each prime integer $n_0$ coprime to $a_1, \dots, a_k$. 
This completes a proof of Lemma~\ref{S_K_intersection}. 
\end{proof}

\bigskip
\section{Realization Property}
\label{proof}

In this section, 
we apply Lemmas~\ref{separation} and \ref{S_K_intersection} to establish Theorem~\ref{realization}. 

\begin{thm_realization}[Realization Theorem]
Let $K$ be a hyperbolic knot. 
Then any finite subset $\mathcal{R} = \{ r_1, \ldots, r_n\} \subset \mathbb{Q}$
is realized by $\mathcal{S}_K(g)$ for some element $g \in [G(K), G(K)] \subset G(K)$ 
whenever the complement of $\mathcal{R}$ contains neither a Seifert surgery slope nor two reducing surgery slopes.  
\end{thm_realization}

\begin{proof}
Assume that $\mathcal{R} = \emptyset$. 
By the assumption $\mathbb{Q} - \mathcal{R} = \mathbb{Q}$ does not contain a Seifert  surgery slope, 
in particular, it does not contain a cyclic surgery slope. 
Then Theorem~\ref{hyperbolic_persistent_[G,G]_noncyclic} shows that we have an element $g \in [G(K), G(K)] \subset G(K)$ with $\mathcal{S}_K(g) = \emptyset = \mathcal{R}$. 
So in the following we assume $\mathcal{R} = \{ r_1, \dots, r_n \}$ is not empty.  
By the assumption the complement of $\mathcal{R}$ does not contain a Seifert surgery slope. 

Take a hyperbolic surgery slope $s$ (i.e. $K(s)$ is hyperbolic) which does not belong to $\mathcal{R}$, and set $\mathcal{S} = \{ s \}$. 
Then apply Lemma~\ref{separation} to find an element $g_1 \in [G(K), G(K)] \subset G(K)$ such that 
$\mathcal{R} \subset \mathcal{S}_K(g_1) \subset \mathbb{Q} - \mathcal{S}$, in particular 
\[
\{ r_1, \ldots, r_n \} \subset \mathcal{S}_K(g_1).
\]

If $\mathcal{S}_K(g_1) = \{ r_1, \ldots, r_n \}$, 
then $g_1$ is a desired element of $G(K)$. 

Assume that $\mathcal{S}_K(g_1) - \{ r_1, \ldots, r_n \}$ is not empty. 
Since $\mathcal{S}_K(g_1)$ is a finite subset of $\mathbb{Q}$, 
we may write 
\[
\mathcal{S}_K(g_1) - \{ r_1, \ldots, r_n \} =  \{ s_1, \ldots, s_m\}.
\]
By the assumption of Theorem~\ref{realization}, 
$s_i$ is not a Seifert surgery slope.

\begin{claim}
\label{g1g2}
There exists a non-trivial element $g_2 \in [G(K), G(K)]$ such that 
\[
\mathcal{S}_K(g_1) \cap \mathcal{S}_K(g_2) = \{ r_1, \ldots, r_n \}.
\]
Furthermore, if $g_1$ is peripheral, then $g_2$ is non-peripheral. 
\end{claim}

\begin{proof} 
We apply Lemma~\ref{separation} again to get an element $g_2 \in [G(K), G(K)] \subset G(K)$ which satisfies
\[
\{r_1, \ldots, r_n \} \subset \mathcal{S}_K(g_2) \subset \mathbb{Q} - \{ s_1, \ldots, s_m \}.
\]
Then 
\[
\mathcal{S}_K(g_1) \cap \mathcal{S}_K(g_2) = \{ r_1, \ldots, r_n \}.
\]

Suppose that $g_1$ is peripheral, i.e. it is conjugate to a power of $\mu^p\lambda^q$ for some coprime integers $p, q$. 
Then we will prove that $g_2$ is non-peripheral. 
Since $g_1 \in  [G(K), G(K)]$, it is conjugate to $\lambda^{\ell_1}$ for some non-zero integer $\ell_1$.   
Let us take a slope $s_1 \in \mathcal{S}_K(g_1)$, which is not a finite surgery slope. 
Note that $g_1 \in \langle \! \langle s_1 \rangle\!\rangle$, 
and hence $\lambda^{\ell_1} \in \langle \! \langle s_1 \rangle\!\rangle$. 
Since $s_1$ is not a finite surgery slope, 
Proposition~\ref{slope} shows that $s_1 = 0$. 
On the other hand, 
since $\mathcal{S}_K(g_2) \subset \mathbb{Q} - \{ s_1, \ldots, s_m \}$, 
$g_2 \not\in  \langle \! \langle s_1 \rangle\!\rangle =  \langle \! \langle 0 \rangle\!\rangle$. 
Recall that $g_2 \in [G(K), G(K)]$. 
So if $g_2$ is peripheral, then $g_2$ is also conjugate to $\lambda^{\ell_2}$ for some non-zero integer $\ell_2$, 
and hence $g_2 \in  \langle \! \langle 0 \rangle\!\rangle$.  
This is a contradiction. 
Hence, $g_2$ is a non-peripheral element. 
\end{proof}

\medskip

Recall that $g_i  \in [G(K), G(K)]$ and 
$\mathbb{Q} - \mathcal{S}_K(g_i) \subset \mathbb{Q} - \mathcal{R}$ 
contains neither Seifert (and hence finite) surgery slope nor two reducing surgery slope by the assumption for $i = 1, 2$. 

If $g_1$ is non-peripheral, 
then following Lemma~\ref{S_K_intersection} 
there exist integers $\ell$ and $m_0$ such that 
\[
\mathcal{S}_K(g_1^{\ell} g_2^{m_0}) = \mathcal{S}_K(g_1) \cap \mathcal{S}_K(g_2). 
\]
Then Claim~\ref{g1g2} shows 
\[
\mathcal{S}_K(g_1^{\ell} g_2^{m_0}) = \{ r_1, \dots, r_n \}.
\]
Note that $\ell$ may be arbitrarily large; see Lemma~\ref{S_K_intersection}.

If $g_1$ is peripheral, then as shown by Claim~\ref{g1g2},  
$g_2$ is non-peripheral. 
Then following Claim~\ref{g1g2} and  Lemma~\ref{S_K_intersection} 
there exist integers $\ell_0$ and $m$ such that 
\[
\mathcal{S}_K(g_1^{\ell_0} g_2^m) = \mathcal{S}_K(g_1) \cap \mathcal{S}_K(g_2) = \{ r_1, \dots, r_n \}. 
\]
Note that $m$ may be arbitrarily large. 

Put $g = g_1^{\ell} g_2^{m_0}$ (if $g_1$ is non-peripheral) 
or $g_1^{\ell_0} g_2^m$ (if $g_1$ is peripheral). 
Then, $\mathcal{S}_K(g) = \mathcal{R}$ as desired, 
and since $g_1$ and $g_2$ belong to $[G(K), G(K)]$, 
$g\in [G(K), G(K)]$ as well. 
This completes a proof of Theorem~\ref{realization}. 
\end{proof}

\medskip 

\begin{remark}
\label{realization_non-peripheral}
\begin{enumerate}
\renewcommand{\labelenumi}{(\arabic{enumi})}
\item 
An element $g \in [G(K), G(K)]$ with $\mathcal{S}_K(g) = \emptyset$ is necessarily 
non-peripheral. 
Indeed, if $g$ is conjugate to a power of a slope element $\mu^p\lambda^q$, 
then $(p, q) = (0, \pm 1)$ and $\mathcal{S}_K(g)$ contains $p/q = 0$, contradicting the assumption.

\item
Suppose that $\mathcal{R}$ is a non-empty finite subset of $\mathbb{Q}$. 
Then as in the proof of Theorem~\ref{realization}, 
$\mathcal{R}$ is realized by 
$\mathcal{S}_K(g)$ for 
some element $g = g_1^{\ell} g_2^{m_0}$ \(if $g_1$ is non-peripheral\) or $g_1^{\ell_0} g_2^m$ \(if $g_1$ is peripheral\). 
Note that we may choose integers $\ell$ and $m$ arbitrarily large. 
So by Lemma~\ref{product_non-peripheral} and Remark~\ref{commute} 
we may take 
$g = g_1^{\ell} g_2^{m_0}$ \(if $g_1$ is non-peripheral\) 
is non-peripheral, 
similarly 
$g_1^{\ell_0} g_2^m$ \(if $g_1$ is peripheral\) is non-peripheral. 
\end{enumerate}
\end{remark}

\medskip

The proof of Theorem~\ref{realization} suggests that $\mathcal{R}$ is realized by $\mathcal{S}_K(\alpha_i)$ for infinitely many elements $\alpha_1, \alpha_2, \dots$. 
However, they may be conjugate or they belong to a single cyclic subgroup of $G(K)$. 
In the next section we will prove that we may take infinitely many, mutually non-conjugate elements $\alpha_i$ so that $\mathcal{S}_K(\alpha_i) = \mathcal{R}$, 
but each $\alpha_i$ is not conjugate to any power of $g$. 

\bigskip

\section{Elements $\alpha, \beta$ with $\mathcal{S}_K(\alpha) = \mathcal{S}_K(\beta)$} 
\label{identical trivialization}

In this section we investigate elements $h \in G(K)$ such that $\mathcal{S}_K(h) = \mathcal{S}_K(g)$ for a given element $g \in G(K)$. 
Since $\mathcal{S}_K$ is a class function,  
$\mathcal{S}_K(g) = \mathcal{S}_K(\alpha^{-1} g \alpha)$ for any $\alpha$. 

For non-conjugate elements, 
we begin by providing an obvious example. 

\begin{proposition}
\label{identical_S_K_power}
Let $K$ be a hyperbolic knot.  
Then for any non-trivial element $g$ we have the following.
\begin{enumerate}
\renewcommand{\labelenumi}{(\arabic{enumi})}
\item
$\mathcal{S}_K(g) = \mathcal{S}_K(g^n)$ for any $n \ne 0$ if $K$ has no torsion surgery. 
\item
$g^m$ is conjugate to $g^n$ if and only if $m = n$
\end{enumerate}
\end{proposition}

\begin{proof}
(1) Since $K$ has no torsion surgery slope, 
Proposition~\ref{S_K(g)_S_K(g^n)} shows that $\mathcal{S}_K(g) = \mathcal{S}_K(g^n)$ for all $n \ne 0$. 

(2) 
Recall that $G(K)$ has no torsion element. 
We divide the argument into two cases depending upon $g$ belongs to $[G(K), G(K)]$ or not. 

Assume first that $g \not\in [G(K), G(K)]$, i.e. it is homologically non-trivial, 
then obviously $g^m$ and $g^n$ are not conjugate when $m \ne n$. 

Let us suppose that $g \in [G(K), G(K)]$.  
We first observe that $\mathrm{scl}_{G(K)}(g) > 0$.  
If $g$ is non-peripheral, 
then following Theorem~\ref{scl_bound} and Lemma~\ref{monotonicity} we have $\mathrm{scl}_{G(K)}(g) > 0$. 
If $g$ is peripheral, 
then $g$ is conjugate to $\lambda^k$ for some $k \ne 0$, where $\lambda$ is a longitude of $K$. 
Then by Lemma~\ref{scl_g^k} and Remark~\ref{scl_longitude} $\mathrm{scl}_{G(K)}(g) = |k| \mathrm{scl}_{G(K)}(\lambda) \ge |k|/2 > 0$. 
Assume that $g^m$ and $g^n$ are conjugate. 
Then since the stable commutator length is invariant under conjugation,  
we have $\mathrm{scl}_{G(K)}(g^m) = \mathrm{scl}_{G(K)}(g^n)$. 
Following Lemma~\ref{scl_g^k}, this implies $n = \pm m$. 
We exclude the possibility $n = -m$. 
Assume for a contradiction that $g^m$ and $g^{-m}$ are conjugate, 
i.e. $g^m = h^{-1}g^{-m}h$ for some $h$. 
Then $g^{2km}= g^{km}g^{km} = g^{km}h^{-1}g^{-km}h =[g^{km}, h^{-1}]$
so $\mathrm{cl}_{G(K)}(g^{2km})=1$ for all $k >0$, 
hence $\mathrm{scl}_{G(K)}(g)=0$, contradicting $\mathrm{scl}_{G(K)}(g) >0$.
\end{proof}

The aim of this subsection is to prove the following theorem which requires an element $g \in G(K)$ to be non-peripheral, 
but we may take infinitely many elements $h$ with $\mathcal{S}_K(h) = \mathcal{S}_K(g)$ so that they are mutually non-conjugate and 
not conjugate to any power of $g$.

\begin{thm_non_rigid}
Let $K$ be a hyperbolic knot without torsion surgery slope. 
For any non-peripheral element $g \in [G(K), G(K)]$ there are infinitely many, non-conjugate elements 
$\alpha_m \in G(K)$ which enjoy the following. 
\begin{enumerate}
\renewcommand{\labelenumi}{(\arabic{enumi})}
\item
$\mathcal{S}_K(\alpha_m) = \mathcal{S}_K(g)$, and
\item
$\alpha_m$ is not conjugate to any power of $g$, 
in particular, $\alpha_m$ does not lie in the cyclic group generated by $g$. 
\end{enumerate}
\end{thm_non_rigid}

Before proving this theorem we recall that a group $G$ is \textit{residually finite\/} if for each non-trivial element $g$ in $G$, 
there exists a normal subgroup of finite index not containing $g$.  
This is equivalent to say that for every $1\neq g \in G$ there exists a homomorphism 
$\varphi \colon G \rightarrow F$ to some finite group $F$ such that $\varphi(g)\neq 1$. 

It follows from \cite{Hem_residual_finite,Pe1,Pe2,Pe3} that the fundamental group of every compact $3$--manifold is residually finite.
For later convenience, 
we slightly strengthen this to the following form.

\begin{proposition}
\label{RF_strong}
Let $G$ be the fundamental group of a compact $3$--manifold. 
Then for any finite family of non-trivial elements $g_1, \ldots, g_n \in G$, 
 there exists a homomorphism 
$\varphi \colon G \rightarrow F$ to some finite group $F$ such that $\varphi(g_i)\neq 1$. 
\end{proposition}

\begin{proof}
Since $G$ is residually finite, 
we have a homomorphism $\varphi_i \colon G \to F_i$ to a finite group $F_i$ such that $\varphi_i(g_i) \ne 1$ for $i = 1, \ldots, n$. 
Consider a homomorphism $\varphi$ 
from $G$ to  a finite group $F = F_1 \times \cdots \times F_n$ such that 
$\varphi(g) = (\varphi_1(g), \ldots, \varphi_n(g)) \in F$. 
Then $\varphi$ is a desired homomorphism from $G$ to the finite group $F$.  
\end{proof}

\medskip

\begin{proof}[Proof of Theorem~\ref{non_rigid}]
(1) Let $g$ be a non-peripheral element in $[G(K), G(K)]$. 
Since $\mathcal{S}_K(g)$ is a finite set, 
we may take a hyperbolic surgery slope $s$ so that $s \ne 0$ and $s \not\in \mathcal{S}_K(g)$, i.e. $p_s(g) \ne 1$. 

Since $G(K)$ is residually finite, 
following Proposition~\ref{RF_strong} we have an epimorphism $\varphi \colon G(K) \to F$ from 
$G(K)$ to a finite group $F$ such that $\varphi(g) \ne 1$ and $\varphi([g, s]) \ne 1$. 
(Since $g$ is non-peripheral and $s$ is peripheral, they do not commute and $[g, s] \ne 1$. 
Actually, if a non-trivial element $g$ commutes with a peripheral element $s$, 
then the images of $g$ by a holonomy representation is also a parabolic element fixing the same fixed point in the sphere at infinity 
$\partial \mathbb{H}^3$. This means $g$ is also peripheral and contradicts the assumption of $g$. 
See also \cite[Theorem 1]{Sim2}. )

Then choose  (and fix) an integer $p > 1$ so that $\varphi(g^p) = \varphi(g)^p = 1$ in $F$. 

Put $h = s^{-1} g s$ and consider elements 
\[
g^{p+pm-1} s^{-1} g^{-p+1} s = g^{p+pm-1} h^{-p+1} \in G(K), 
\]
where $m \in \mathbb{Z}$. 

Note that since $g \in [G(K), G(K)]$, $h$ also belongs to $[G(K), G(K)]$. 
It follows from Lemma~\ref{S_K_intersection} and Remark~\ref{shrink_torsion_free} 
that for a given non-zero integer $-p+1$,  
there exists a constant $N = N_{-p+1}> 0$ such that for any integer $m \ge N$, 
we have 
\[
\mathcal{S}_K(g^{p + pm -1} h^{-p+1}) = \mathcal{S}_K(g) \cap \mathcal{S}_K(h),  
\]
which coincides with $\mathcal{S}_K(g)$, 
because $\mathcal{S}_K(h) = \mathcal{S}_K(s^{-1} g s) = \mathcal{S}_K(g)$. 

In the following for integers $m \ge N$, 
we put 
\[
\alpha_m = g^{p+pm-1} s^{-1} g^{-p+1} s = g^{p+pm-1} h^{-p+1} \in G(K), 
\]
for which $\mathcal{S}_K(\alpha_m) = \mathcal{S}_K(g)$.  

\begin{claim}
\label{infinite_conjugacy_m}
There are infinitely many integers $m \ge N$ such that 
$\alpha_m = g^{p + pm -1} h^{-p+1}$ are mutually non-conjugate elements in $G(K)$. 
\end{claim}

\begin{proof}
Recall that $g$ is a non-peripheral element in $[G(K), G(K)]$, 
and $h = s^{-1} g s \in [G(K), G(K)]$. 
(Although $g$ is non-peripheral, 
since $h$ is also non-peripheral, 
we cannot apply  Lemma~\ref{product_non-peripheral} (2); compare with Lemma~\ref{infinite_conjugacy}.) 

Since $g$ is non-peripheral, $\textrm{scl}_{G(K)}(g)>0$ (Theorem~\ref{scl_bound} and Lemma~\ref{monotonicity}). 
Furthermore, since $g, h \in [G(K), G(K)]$, 
by Lemma~\ref{scl_product} we have
\[
\mathrm{scl}_{G(K)}(g^{p + pm -1} h^{-p+1}) \ge (p + pm -1) \mathrm{scl}_{G(K)}(g) -\mathrm{scl}_{G(K)}(h^{-p+1}) -\frac{1}{2}.
\]
Hence $\lim_{m\to \infty} \mathrm{scl}_{G(K)}(g^m h^{n_0}) \to \infty$. 
Since stable commutator length is invariant under conjugation, 
this shows that $\{ \alpha_m \} = \{ g^{p + pm -1} h^{-p+1} \}$ contains infinitely many mutually non-conjugate elements. 
\end{proof}

\medskip

(2) Let us prove: 

\begin{claim}
\label{not_conjugate_power}
$\alpha_m$ is not conjugate to $g^k$ for any integer $k$. 
\end{claim}

\begin{proof}
Assume for a contradiction that 
$\alpha_m$ is conjugate to $g^{\ell_m}$ for some integer $\ell_m$.
Then 
$p_s(\alpha_m) = p_s(g^{p+pm-1} s^{-1} g^{-p+1} s)$ is also conjugate to $p_s(g)^{\ell_m}$ in $\pi_1(K(s))$. 
Since 
\[
p_s(g^{p+pm-1} s^{-1} g^{-p+1} s) 
= p_s(g)^{p+pm-1} p_s(s)^{-1} p_s(g)^{-p+1} p_s(s) 
= p_s(g)^{pm}, 
\] 
$p_s(g)^{pm}$ is conjugate to $p_s(g)^{\ell_m}$ in $\pi_1(K(s))$. 

Recall that we choose the slope $s$ so that it is a hyperbolic surgery slope, 
and the non-peripheral element $g$ satisfies $p_s(g) \ne 1$. 
Thus by Theorem~\ref{scl_bound}, we have $\mathrm{scl}_{\pi_1(K(s))}(p_s(g)) > 0$. 

Hence, 
we have the equality 
\begin{align*}
|pm| \mathrm{scl}_{\pi_1(K(s))}(p_s(g)) 
&= 
\mathrm{scl}_{\pi_1(K(s))}(p_s(g)^{pm}) 
= 
\mathrm{scl}_{\pi_1(K(s))}(p_s(g)^{\ell_m}) \\
&=
|\ell_m| \mathrm{scl}_{\pi_1(K(s))}(p_s(g)),
\end{align*}
which implies $\ell_m = \pm  pm$.

On the other hand, 
$\varphi(\alpha_m) = \varphi(g^{p+pm-1} s^{-1} g^{-p+1} s)$ is conjugate to $\varphi(g^{\ell_m}) = \varphi(g^{\pm pm})$. 
By the choice of $p$, 
$\varphi(g^{p+pm-1} s^{-1} g^{-p+1} s) = \varphi(g^{-1} s^{-1} g s)  \ne 1$, 
but $\varphi(g^{\pm pm}) = 1$. 
This is a contradiction. 
So $\alpha_m$ is not conjugate to $g^k$ for any integer $k$. 
\end{proof}

Claims~\ref{not_conjugate_power} and \ref{infinite_conjugacy_m} complete a proof of Theorem~\ref{non_rigid}.
\end{proof}

\medskip

\begin{remark}
$\alpha_m$ and $\alpha_n$ $(m < n)$ do not commute. 
In particular, $\alpha_n$ is not any power of $\alpha_m$. 
\end{remark}

\begin{proof}
Note that 
\[
\alpha_n = g^{p+pn -1} h^{-p+1} = g^{p(n-m)} g^{p+pm -1} h^{-p+1} = g^{p(n-m)} \alpha_m.
\] 

Assume that $\alpha_m \alpha_n = \alpha_n \alpha_m$. 
Then we have $\alpha_m (g^{p(n-m)}\alpha_m) =  (g^{p(n-m)}\alpha_m) \alpha_m$, 
i.e. $\alpha_m g^{p(n-m)} = g^{p(n-m)} \alpha_m$. 
This means that $\alpha_m$ belongs to the centralizer of $g^{p(n-m)}$. 
Since $g$ is non-peripheral,  
so $g^{p(n-m)}$ is also non-peripheral, and hence its centralizer is a cyclic group generated by $g$. 
Hence,  $\alpha_m = g^k$ for some integer $k$. 
This contradicts $\alpha_m$ not being conjugate to any powers of $g$. 
\end{proof}

\section*{Acknowledgements}
The authors would like to thank the referee
for careful reading and valuable comments.

\section*{Funding}
This work was supported by Japan Society for the Promotion of Science, KAKENHI
[Grant Number 19K03490, 21H04428, 23K03110 to T. I.,  
25K07018, 21H04428, 23K03110, 23K20791 to K. M.,  and 
20K03587, 25K07004 to M. T.].


\end{document}